\documentclass[11pt,reqno]{amsart}
\usepackage{graphicx,amsmath,amsthm,amssymb}
\usepackage{verbatim}
\usepackage{enumerate}

\usepackage{xcolor}
\definecolor{darkgreen}{rgb}{0,0.75,0}
\definecolor{darkred}{rgb}{0.75,0,0}
\definecolor{darkmagenta}{rgb}{0.5,0,0.5}

\usepackage[colorlinks,citecolor=darkgreen,linkcolor=darkred,urlcolor=darkmagenta]{hyperref}

\newcommand{\arxiv}[1]{{\tt \href{http://arxiv.org/abs/#1}{arXiv:#1}}}
\newcommand{\mr}[1]{{\tt \href{http://www.ams.org/mathscinet-getitem?mr=#1}{MR#1}}}

\newcommand{\floor}[1]{\left\lfloor {#1} \right\rfloor}

\newcommand{\ghost}[1]{\textcolor{white}{#1}}
\newcommand{\old}[1]{}
\newcommand{\moniker}[1]{{\em (#1)}}
\newcommand{\dmoniker}[1]{(#1)}
\newcommand{\abs}[1]{{\left\vert\kern-0.25ex #1
    \kern-0.25ex\right\vert}}
\newcommand{\eqindist}{\stackrel{d}{=}}				

\newcommand{\ttk}{2^{2^{k}}}

\DeclareRobustCommand{\SkipTocEntry}[5]{}

\newtheorem{theorem}{Theorem}[section]
\newtheorem{prop}[theorem]{Proposition}

\newtheorem{lemma}[theorem]{Lemma}
\newtheorem{lem}[theorem]{Lemma}
\newtheorem{corollary}[theorem]{Corollary}
\newtheorem{conjecture}[theorem]{Conjecture}
\newtheorem{open}[theorem]{Open problem}
\newtheorem{proposition}[theorem]{Proposition}

\theoremstyle{remark}
\newtheorem*{remark}{Remark}
\newtheorem{example}{Example}
\newtheorem{rem}{Remark}

\numberwithin{counter}{section}

\theoremstyle{definition}
\newtheorem{definition}[theorem]{Definition}

\def\liminf{\mathop{\rm lim\,inf}\limits}

\def\xx{\mathbf{x}}

\def\00{\mathbf{0}}

\def\Var{\mathop{\mathrm{Var}}}

\def\zero{\mathbf{0}}
\def\one{\mathbf{1}}

\def\N{\mathbb{N}}
\def\Z{\mathbb{Z}}

\def\R{\mathbb{R}}
\def\EE{\mathbb{E}}
\def\PP{\mathbb{P}}
\def\eps{\epsilon}
\def\var{\operatorname{Var}}
\def\aut{\operatorname{Aut}}
\newcommand\norm[1]{\left\lVert#1\right\rVert} 

\begin{document}

\title{The divisible sandpile at critical density}

\author[Levine, Murugan, Peres, Ugurcan]{Lionel Levine, Mathav Murugan, Yuval Peres and Baris Evren Ugurcan}

\address{Lionel Levine, Department of Mathematics, Cornell University, Ithaca, NY 14853, USA. {\tt \url{http://www.math.cornell.edu/~levine}}}
\address{Department of Mathematics, University of British Columbia and Pacific Institute for the Mathematical Sciences, Vancouver, BC V6T 1Z2, Canada. mathav@math.ubc.ca}
\address{Yuval Peres, Microsoft Research, Redmond, WA 98052, USA. peres@microsoft.com}
\address{Baris Evren Ugurcan, Department of Mathematics, Cornell University, Ithaca, NY 14853, USA. beu4@cornell.edu}

\date{July 28, 2015}
\thanks{The first author was supported by \href{http://www.nsf.gov/awardsearch/showAward?AWD_ID=1243606}{NSF DMS-1243606} and a Sloan Fellowship.}

\begin{abstract}
The divisible sandpile starts with i.i.d.\ random variables (``masses'') at the vertices of an infinite, vertex-transitive graph, and redistributes mass by a local toppling rule in an attempt to make all masses $\leq 1$. The process stabilizes almost surely if $m<1$ and it almost surely does not stabilize if $m>1$, where $m$ is the mean mass per vertex.
The main result of this paper is that in the critical case $m=1$, if the initial masses have finite variance, then the process almost surely does not stabilize. To give quantitative estimates on a finite graph, we relate the number of topplings to a discrete bi-Laplacian Gaussian field.
\end{abstract}

\maketitle

\section{Introduction}

This paper is concerned with the dichotomy between \emph{stabilizing} and \emph{exploding} configurations in a model of mass redistribution, the \emph{divisible sandpile model}.
The main interest in this model is twofold.
First, it is a natural starting place for the analogous and more difficult dichotomy in the \emph{abelian sandpile model}. Second, the divisible sandpile itself leads to interesting questions in potential theory. For example, under what conditions must a random harmonic function be an almost sure constant? (Lemma~\ref{l.stationaryharmonic} gives some sufficient conditions.)
Both the motivation for this paper and many of the proof techniques are directly inspired by the work
 of Fey, Meester and Redig \cite{FMR}.

By a \emph{graph} $G=(V,E)$ we will always mean a connected, locally finite and undirected graph with vertex set $V$ and edge set $E$. We write $x \sim y$ to mean that $(x,y) \in E$, and $\deg(x)$ for the number of $y$ such that $x \sim y$.
A \emph{divisible sandpile configuration} on $G$ is a function $s: V \to \R$. We refer to $s(x)$ as an amount of `mass' present at vertex $x$; a negative value of $s(x)$ can be imagined as a `hole' waiting to be filled by mass. A vertex $x \in V$ is called \emph{unstable} if $s(x)> 1$. An unstable vertex $x$ \emph{topples} by keeping mass $1$ for itself and distributing the excess $s(x) -1$ equally among its neighbors $y \sim x$.
At each discrete time step, all unstable vertices topple simultaneously. (This parallel toppling assumption is mainly for simplicity; in Section \ref{s.maindefs} we will relax it.)
The following trivial consequence of the toppling rule is worth emphasizing: if for a particular vertex $x$ the inequality $s(x) \geq 1$ holds at some time, then it holds at all later times.

Note that the entire system evolves deterministically once an initial condition $s$ is fixed.
The initial $s$ can be deterministic or random; below we will see one example of each type.
Write $u_n(x)$ for the total amount of mass emitted before time $n$ from $x$ to one of its neighbors. (By the symmetry of the toppling rule, it does not matter which neighbor.) This quantity increases with $n$, so $u_n \uparrow u$ as $n \uparrow \infty$ for a function $u : V \to [0,\infty]$. We call this function $u$ the \emph{odometer} of $s$. Note that if $u(x)=\infty$ for some $x$, then each neighbor of $x$ receives an infinite amount of mass from $x$, so $u(y) = \infty$ for all  $y \sim x$.  We therefore have the following dichotomy:
	\begin{align*} \text{Either \; $u(x)<\infty$ for all $x \in V$,} \\ \text{or \; $u(x) = \infty$ for all $x \in V$.}  \end{align*}
In the former case we say that $s$ \emph{stabilizes}, and in the latter case we say that $s$ \emph{explodes}.

 The following theme repeats itself at several places: \emph{The question of whether $s$ stabilizes depends not only on $s$ itself but also on the underlying graph}.  For instance, fixing a vertex $o$, we will see that the divisible sandpile
	\[ s(x) = \begin{cases} 1 & x \neq o \\ 2 & x=o \end{cases} \]
stabilizes on $G$ if and only if the simple random walk on $G$ is transient (Lemma~\ref{l.diracplusone}).

Our main result treats the case of initial masses $s(x)$ that are independent and identically distributed (i.i.d.) random variables with finite variance. Write $\EE s$ and $\Var s$ for the common mean and variance of the $s(x)$.  The mean $\EE s$ is sometimes called the \emph{density} (in the physical sense of the word, mass per unit volume). Because sites topple when their mass exceeds $1$, intuition suggests that the density should be the main determiner of whether or not $s$ stabilizes: the higher the density, the harder it is to stabilize.  Indeed, we will see that $s$ stabilizes almost surely if  $\EE s <1$  (Lemma \ref{l.stable}) and explodes almost surely if $\EE s >1$ (Lemma \ref{l.notstab}).
Our main result addresses the critical case $\EE s =1$.

\begin{theorem} \label{main}
Let $s$ be an i.i.d.\ divisible sandpile on an infinite, vertex-transitive graph, with $\EE s=1$ and $0 < \var s < \infty$.  Then $s$ almost surely does not stabilize.
\end{theorem}

Our theme that `stabilizability depends on the underlying graph' repeats again in the proof of Theorem \ref{main}.
The proof splits into three cases depending on the graph. The cases in increasing order of difficulty are
\begin{itemize}
 \item recurrent (Lemma \ref{l.recurrent}). Examples: $\Z, \Z^2$.
 \item transient with $\sum_{x \in V} g(o,x)^2 = \infty$  (Section \ref{s.strans}). Examples: $\Z^3,\Z^4$.
 \item transient with $\sum_{x \in V} g(o,x)^2 <\infty$ (Section \ref{s.dtrans}). Examples: $\Z^d$ with $d\ge 5$.
\end{itemize}
Here $g$ denotes Green's function: $g(o,x)$ is the expected number of visits to $x$ by a simple random walk started at $o$. The reason for the order of difficulty is that `stabilization is harder in lower dimensions,' in a sense formalized by Theorem~\ref{p.torus} below.

\subsection{Potential theory of real-valued functions}

The \emph{graph Laplacian} $\Delta$ acts on functions $u: V \to \R$ by
 	\begin{equation} \label{e.thelaplacian} \Delta u(x) = \sum_{y \sim x} (u(y) - u(x)). \end{equation}
Using a `least action principle' (Proposition \ref{lap}), the question of whether a divisible sandpile stabilizes can be reformulated as a question in potential theory:

  \begin{quote}
  Given a function $s: V \to \R$,
 does there exist a nonnegative function
 $u : V \to \R$ such that $s + \Delta u \le 1$ pointwise?
 \end{quote}

  \subsection{Potential theory of integer-valued functions}
In the related \emph{abelian sandpile model},
configurations are integer-valued functions $s : V \to \Z$. We think of $s(x)$ as a number of particles present at $x$.  A vertex $x \in V$ is \emph{unstable} if it
 has at least $\deg(x)$ particles. An unstable site $x$ \emph{topples} by sending one particle to each of its $\deg(x)$ neighbors.  This model also has a dichotomy between stabilizing ($u<\infty$) and exploding ($u \equiv \infty$), which can be reformulated as follows:

 \begin{quote}
 Given a function $s: V \to \Z$, does there exist a nonnegative
function $u: V \to \Z$ such that $s + \Delta u \le \deg - 1$ pointwise?
\end{quote}

The restriction that $u$ must be integer-valued introduces new difficulties that are not present in the divisible sandpile model. The first step in the proof of Theorem~\ref{main} is to argue that if $\EE s =1$ and $s$ stabilizes, then it necessarily stabilizes to the all $1$ configuration. 
This step fails for the abelian sandpile except in dimension~$1$.
Indeed, a result analogous to Theorem~\ref{main} does hold for the abelian sandpile when the underlying graph is $\Z$ \cite[Theorem~3.2]{FMR}, but no such result can hold in higher dimensions:
The density $\EE s$ alone is not enough to determine whether an abelian sandpile $s$ on $\Z^d$ stabilizes, if $d < \EE s < 2d-1$ (see \cite[Section 5]{FR}, \cite[Theorem 3.1]{FMR} and \cite[Proposition 1.4]{FLP}; the essential idea in these arguments arose first in bootstrap percolation \cite{Ent,Sch}).

We would like to highlight an open problem: Given a probability distribution $\mu$ on $\Z$ (say, supported on $\{0,1,2,3,4\}$ with rational probabilities) is it algorithmically decidable whether the i.i.d.\ abelian sandpile on $\Z^2$ with marginal $\mu$ stabilizes almost surely?

\subsection{Quantitative estimates and bi-Laplacian field}

For a \emph{finite} connected graph $G=(V,E)$, the divisible sandpile $s : V \to \R$ stabilizes if and only if $ \sum_{x \in V} s(x) \le \abs{V}$.
Our next result gives the order of the odometer in a critical case when this sum is exactly $|V|$.
Specifically, to formalize the idea that `stabilization is harder in lower dimensions,' we take an identically distributed Gaussian initial condition on the discrete torus $\Z_n^d$, conditioned to have total mass $n^d$.  The expected odometer can be taken as an indication of difficulty to stabilize: How much mass must each site emit on average? According to equation \eqref{e.phi} below, the expected odometer tends to $\infty$ with $n$ in all dimensions (reflecting the failure to stabilize on the infinite lattice $\Z^d$) but it
decreases with dimension.

\begin{theorem} \label{p.torus}
Let $( \sigma(\mathbf x))_{\mathbf  x \in \Z^d_n}$ be i.i.d.\ $N(0,1)$, and consider the divisible sandpile
	\[ s_{d,n} (\mathbf  x) = 1 + \sigma(\mathbf  x) - \frac{1}{n^d} \sum_{\mathbf  y \in \Z^d_n} \sigma(\mathbf  y). \]
Then $s_{d,n}:\Z^d_n \to \R$ stabilizes to the all $1$ configuration, and there exists a constant $C_d$ such that the odometer $u_{d,n}$ satisfies
\[
C_d^{-1} \phi_d(n) \le \EE u_{d,n}(x) \le C_d \phi_d(n)
\]
for all $n \ge 2$, where $\phi_d$ is defined by
\begin{equation} \label{e.phi}
 \phi_d(n):= \begin{cases}
                          n^{3/2},  &  d=1 \\
                          n,  & d=2 \\
                          n^{1/2}, & d=3 \\
                              \log n,   & d=4  \\
                             (\log n)^{1/2}, & d \ge 5.                         \end{cases}
\end{equation}
\end{theorem}

The first step in computing these orders is proving an equality in law between the odometer $u_{d,n}$ and a certain `discrete bi-Laplacian Gaussian field' shifted to have minimum value $0$.  This equality in law actually holds for any finite connected graph, as detailed in the next proposition.

\begin{proposition} \label{t.finite}
Let $G= (V, E)$ be a finite connected graph.
Let $( \sigma(x))_{x \in V}$ be i.i.d.\ $N(0,1)$, and consider the divisible sandpile
	\[ s(x) = 1 + \sigma(x) - \frac{1}{\abs{V}} \sum_{y \in V} \sigma(y). \]
Then $s$ stabilizes to the all $1$ configuration, and the distribution of its odometer
$u : V \to [0,\infty)$ is
	\[ (u(x))_{x \in V} \eqindist \left(  \eta(x) - \min \eta \right)_{x \in V} \]
where the $\eta(x)$ are jointly Gaussian with mean zero and covariance
	\[ \mathbb{E}[\eta(x) \eta(y)] = \frac{1}{\deg(x)\deg(y)}\sum_{z \in V} g(z,x) g(z,y) \]
	where $g$ is defined by  $g(x,y) = \frac{1}{\abs{V}} \sum_{z \in V} g^z(x,y)$ and $g^z(x,y)$ is the expected number of visits to $y$ by the simple random walk started at $x$ before hitting $z$.
\end{proposition}

Proposition~\ref{t.finite} suggests the possibility of a central limit theorem for the divisible sandpile odometer on $\Z_n^d$: We believe that if $\sigma$ is identically distributed with zero mean and finite variance, then the odometer, after a suitable shift and rescaling,
converges weakly as $n \to \infty$ to the bi-Laplacian Gaussian field on $\R^d$.

\subsection{Proof ideas}

By conservation of density (Proposition~\ref{p.cden}) the assumption $\EE s = 1$ implies that if $s$ stabilizes then it must stabilize to the all $1$ configuration, so that the odometer $u$ satisfies
	\begin{equation} \label{e.allones} \Delta u = 1-s. \end{equation}
where $\Delta$ is the Laplacian \eqref{e.thelaplacian}. This relation leads to a contradiction in one of three ways:

\begin{itemize}
\item If $G$ is recurrent (examples: $\Z,\Z^2$), then $1+\beta \delta_o$ does not stabilize (Lemma~\ref{l.diracplusone}). By resampling the random variable $s(o)$ we derive a contradiction from \eqref{e.allones}.

\item If $G$ is simply transient (examples: $\Z^3, \Z^4$) then we can attempt to solve \eqref{e.allones} for $u$, writing
	\begin{equation} \label{e.greenseries} u(y) = \sum_x g(x,y) (s(x)-1) \end{equation}
where $g$ is Green's function. The sum on the right side diverges a.s.\ if taken over all $x \in V$ (since $g(\cdot,y)$ is not square-summable) but we can stabilize $s$ in nested finite subsets $V_n \uparrow V$ instead. The corresponding finite sums, suitably normalized, tend in distribution to a mean zero Gaussian by the Lindeberg central limit theorem, contradicting the nonnegativity of $u(y)$.

\item If $G$ is doubly transient (example $\Z^d$ for $d \geq 5$) then the right side of \eqref{e.greenseries} converges a.s.. The difference between the left and right sides is then a random harmonic function with automorphism-invariant law.  The proof is completed by showing that under mild moment assumptions any such function is an almost sure constant (Lemma~\ref{l.stationaryharmonic}).
\end{itemize}

\subsection{Related work}

The divisible sandpile was introduced in \cite{LP09,LP10} to study the scaling limits of two growth models,
rotor aggregation and internal DLA. The divisible sandpile has also been used as a device for proving an exact mean value property for discrete harmonic functions \cite[Lemma 2.2]{JLS13}.
These works focused on sandpiles with finite total mass on an infinite graph, in which case exploding is not a possibility.
In the present paper we expand the focus to sandpiles with infinite total mass.

The abelian sandpile has a much longer history: it arose in statistical physics as a model of `self-organized criticality' (SOC) \cite{BTW87,Dha90}. The dichotomy between stabilizing and exploding configurations arose in the course of a debate about whether SOC does or does not involve tuning a parameter to a critical value \cite{FR,MQ}.
Without reopening that particular debate, we view the stabilizing/exploding dichotomy as a topic with its own intrinsic mathematical interest.   An example of its importance can be seen in the partial differential equation for the scaling limit of the abelian sandpile on $\Z^2$, which relies on a classification of certain `quadratic' sandpiles according to whether they are stabilizing or exploding \cite{LPS12}.

The Gaussian vector $\eta$ in Proposition~\ref{t.finite} can be interpreted as a discrete bi-Laplacian field. In $\Z^d$ for dimensions $d\geq 5$, Sun and Wu construct another discrete model for the bi-Laplacian field by assigning random signs to each component of the uniform spanning forest \cite{SW13}.

\section{Toppling procedures and stabilization} \label{s.maindefs}

In this section $G=(V,E)$ is a locally finite, connected, undirected graph. Denote by $\mathcal{X}= \mathbb{R}^{V}$ the set of divisible sandpile configurations on $G$.

\begin{definition} \label{dtoppling}
Let $T \subset  [0,\infty)$ be a well-ordered set of \emph{toppling times} such that $0 \in T$ and $T$ is a closed subset of $[0,\infty)$.
A \emph{toppling procedure} is a function
	\begin{align*} T \times V &\to [0,\infty) \\
				(t,x) &\mapsto u_t(x) \end{align*}
satisfying for all $x \in V$
	\begin{enumerate}
	\item $u_0(x)=0$.
	\item $u_{t_1}(x) \leq u_{t_2}(x)$ for all $t_1 \leq t_2$.
	\item If $t_n \uparrow t$, then $u_{t_n}(x) \uparrow u_t(x)$.
	\end{enumerate}
\end{definition}

In the more general toppling procedures considered by Fey, Meester and Redig \cite{FMR}, the assumption that $T$ is well-ordered becomes a ``no infinite backward chain'' condition, but we will not need that level of generality.  See Examples \ref{e.parallel}-\ref{e.stages} below for the three specific toppling procedures we will use.

The interpretation of a toppling procedure is that starting from an initial configuration $s \in \mathcal{X}$, the total mass emitted by a site $x \in V$ to each of its neighbors during the time interval $[0,t]$ is $u_t(x)$, so that the resulting configuration at time $t$ is
	\[ s_t = s + \Delta u_t \]
where $\Delta$ is the graph Laplacian \eqref{e.thelaplacian}.

For $a \in \R$ write $a^+ = \max(a,0)$.  For $t \in T$ write $t^- := \sup \{r \in T \,:\, r<t \}$. Note that $t^- \in T$ since $T$ is closed.

\begin{definition}
A toppling proceedure $u$ is called \emph{legal} for initial configuration $s$ if
	\[ u_t(x) - u_{t^-}(x) \leq \frac{(s_{t^-}(x)-1)^+}{ \operatorname{deg}(x)}\]
for all $x \in V$ and all $t \in T \setminus \{ 0 \}$.
\end{definition}

Thus, in a legal toppling procedure, a site with mass $\leq 1$ cannot emit any mass, while a site with mass $>1$ must keep at least mass $1$ for itself.

\begin{definition}
A toppling procedure $u$ is called \emph{finite} if for all $x \in V$ we have
\begin{equation*}
u_\infty(x) := \lim_{t \to \sup T} u_t(x)
< \infty
\end{equation*}
and \emph{infinite} otherwise.
The limit exists in $[0,\infty]$ since $u_t(x)$ is nondecreasing in $t$.
\end{definition}

Note that if $u$ is a finite toppling procedure, then the limit
	\[ s_\infty := \lim_{t \to \sup T} s_t = s + \lim_{t \to  \sup T} \Delta u_t \]
exists and equals $s + \Delta u_\infty$.

\begin{definition} \label{defstab}
Let $s \in \mathcal{X}$.
A toppling procedure $u$ is called \emph{stabilizing} for $s$ if $u$ is finite and $s_\infty \leq 1$ pointwise.  We say that $s$ \emph{stabilizes} if there exists a stabilizing toppling procedure for $s$.
\end{definition}

Throughout this paper, all inequalities between functions hold pointwise, and we will usually omit the word ``pointwise.''

A basic question arises: For which $s \in \mathcal{X}$ does there exist a stabilizing toppling procedure? For instance, one might expect (correctly) that there is no such procedure for $s \equiv 2$. We can rephrase this question in terms of the set of functions
	\[ \mathcal{F}_s := \{ f: V \to \R \,|\, f \geq 0 \text{ and } s + \Delta f \leq 1 \}. \]
If $u$ is a stabilizing toppling procedure for $s$, then $u_\infty \in \mathcal{F}_s$. Conversely, any $f \in \mathcal{F}_s$ arises from a stabilizing toppling procedure for $s$ simply by setting $T = \{0,1\}$ and $u_1 = f$.  Therefore $s$ stabilizes if and only if $\mathcal{F}_s$ is nonempty.


\begin{prop}\label{lap} \moniker{Least action principle and abelian property}
Let $s \in \mathcal{X}$, and let $\ell$ be a legal toppling procedure for $s$.
\begin{enumerate}[\em (i)]
\item For all $f \in \mathcal{F}_s$,
	\[ \ell_\infty \leq f. \]

\item If $u$ is any stabilizing toppling procedure for $s$, then \[ \ell_\infty \leq u_\infty. \]

\item If $u$ is any legal stabilizing toppling procedure for $s$, then for all $x \in V$,
	\begin{equation} \label{e.LAP} u_\infty(x) = \inf \{f(x) \,|\, f \in \mathcal{F}_s \} \end{equation}
In particular, $u_\infty$ and the final configuration \[ s_\infty = s + \Delta u_\infty \] do not depend on the choice of legal stabilizing toppling procedure $u$.
\end{enumerate}
\end{prop}

\begin{proof}
(i) For $y \in V$ let $\tau_y = \inf \{ t \in T \,:\, \ell_t(y)>f(y) \}$, and suppose for a contradiction that $\tau_y < \infty$ for some $y \in V$.
Since $T$ is well-ordered, the infimum is attained.  Moreover, $\tau_y^- := \sup \{t \in T \,:\, t < \tau_y\}<\tau_y$ by assumption (3) of Definition~\ref{dtoppling}.

Let $\tau = \inf_{y \in V} \tau_y$. Since $T$ is well-ordered, $\tau = \tau_y$ for some $y \in V$. Now at time $\tau^-$, since $u$ is legal for $s$,
	\[ s_{\tau^-}(y) \geq 1 + \operatorname{deg}(y)\left(\ell_\tau(y) - \ell_{\tau^-}(y)\right) > 1 +\operatorname{deg}(y) w(y) \]
where $w = f - \ell_{\tau^-}$. On the other hand,
	\begin{align*} s_{\tau^-}(y) &= (s + \Delta \ell_{\tau^-})(y) \\
						&= (s + \Delta f)(y) - \Delta w (y) \\
						&\leq 1 + \deg(y)w(y) -  \sum_{x \sim y} w(x).
						\end{align*}
It follows that $\sum_{x \sim y} w(x)  < 0$. But for all $x \in V$ we have $\tau^- < \tau \leq \tau_x$, so $w(x) \geq 0$, which yields the required contradiction.

Part (ii) follows from (i), using $f = u_\infty$.

Part (iii) also follows from (i): the function $u_\infty$ simultaneously attains the infimum for all $x\in V$.
\end{proof}

Whenever we need to fix a particular toppling procedure, we will choose one of the following three.

\begin{example} \label{e.parallel}
\dmoniker{Toppling in parallel}
Let $T=\N$. At each time $t\in \N$, all unstable sites of $s_{t-1}$ topple all of their excess mass simultaneously: For all $x \in V$,
	\[ u_t(x) - u_{t-1}(x) = \frac{(s_{t-1}(x)-1)^+}{\operatorname{deg}(x)}. \]
This $u$ is a legal toppling procedure. Two further observations about $u$ will be useful in the proof of Lemma~\ref{c.lap} below. First, if $s$ stabilizes, then $u$ is finite by Proposition~\ref{lap}(ii).
Second, whenever $u$ is finite, $u$ is stabilizing for $s$: indeed, if $u_\infty(x) = \frac{1}{\deg(x)} \sum_{t \in \N} (s_t(x)-1)^+ < \infty$, we have $(s_t(x)-1)^+ \to 0$ as $t \to \infty$ and hence $s_\infty(x) \le 1$.
\end{example}

\begin{example} \label{e.nest}
\dmoniker{Toppling in nested volumes}
Let $V_1 \subset V_2 \subset \ldots$ be finite sets with $\bigcup_{n \geq 1} V_n = V$. Between times $n-1$ and $n$ we topple in parallel to stabilize all sites in $V_n$: Formally, we take $T$ to be the set of all rationals of the form $n-\frac{1}{k}$ for positive integers $n$ and $k$.
For $n \geq 1$ and $k \geq 1$ we set
	\[ u_{n-\frac{1}{k}}(x) - u_{n-\frac{1}{k-1}}(x) = \frac{(s_{n-\frac{1}{k-1}}(x)-1)^+ \cdot \one_{x \in V_n}}{\operatorname{deg}(x)} \]
and
	\[ u_{n}(x) = \lim_{k \to \infty} u_{n-\frac{1}{k}}(x). \]
\end{example}

\begin{example} \label{e.stages}
\dmoniker{Toppling in two stages}
In this procedure we are given a decomposition of the initial configuration into two pieces
	\[ s = s^1 + s^2 \]
where $s^1$ stabilizes and $s^2 \geq 0$. In the first stage we ignore the extra mass $s^2$ and stabilize the $s^1$ piece by toppling in parallel at times $1-\frac1k$ for positive integers $k$, obtaining
	\[ s_1 = s + \Delta u^1_\infty = s^1_\infty + s^2. \]
The condition that $s^1$ stabilizes ensures $u^1_\infty < \infty$, and the condition $s^2 \geq 0$ ensures that all topplings that are legal for $s^1$ are also legal for $s$.  Now we topple $s^1_\infty +s^2$ in parallel at times $2-\frac1k$ for positive integers $k$.
\end{example}

Now we come to a central definition of this paper. Let $s \in \mathcal{X}$.

\begin{definition}
The \emph{odometer} of $s$ is the function $u_\infty : V \to [0,\infty]$ of \eqref{e.LAP}.  If $s$ stabilizes, then its \emph{stabilization} is the configuration
	\[ s_\infty = s + \Delta u_\infty. \]
\end{definition}

If $s$ stabilizes, then its odometer $u_\infty(x)$ is the total amount of mass sent from $x$ to one of its neighbors, in any legal stabilizing toppling procedure for $s$.
If $s$ does not stabilize, then $u_\infty \equiv \infty$:
The odometer is defined as a pointwise infimum, with the usual convention that the infimum of the empty set is $\infty$. Next we observe that the odometer can also be expressed as a pointwise supremum.

\begin{lemma}\label{c.lap}
Let $u_\infty$ be the odometer of $s \in \mathcal{X}$.  Then for all $x \in V$,
  \begin{equation} \label{e.legalsup} u_\infty(x) = \sup \{ \ell_\infty(x) \,|\, \ell \mbox{ is a legal toppling procedure for $s$} \}. \end{equation}
\end{lemma}

\begin{proof}
Denote the right side of \eqref{e.legalsup} by $L(x)$. By Proposition \ref{lap}(i), $L \le u_\infty$. To prove the reverse inequality we will use a particular legal $\ell$, the parallel toppling procedure of Example~\ref{e.parallel}. There are two cases: First, if this $\ell$ is finite, then $\ell$ is stabilizing as well as legal, so $L \geq \ell_\infty = u_\infty$ by Proposition \ref{lap}(iii).

Second, if $\ell$ is not finite, then $\ell_\infty(o)=\infty$ for some $o \in V$. Then for any neighbor $x \sim o$, we have
 \[
  \ell_{t+1}(x) \ge \ell_{t}(o) + s(x)-1 \]
and the right side tends to $\infty$ with $t$, so $\ell_\infty(x)=\infty$.
Since the graph $G$ is connected it follows that $\ell_\infty \equiv \infty$.  In this case, both $L$ and $u_\infty$ are identically $\infty$.
\end{proof}

We pause to record several equivalent conditions for $s$ stabilizing.

\begin{corollary}\label{c.equiv}
Let $s \in \mathcal{X}$ have odometer $u_\infty$. The following are equivalent.
\begin{enumerate}
\item There exists a legal stabilizing toppling procedure for $s$.
\item There exists a stabilizing toppling procedure for $s$.
\item $\mathcal{F}_s \neq \emptyset$.
\item $u_\infty(x) < \infty$ for all $x \in V$.
\item Every legal toppling procedure for $s$ is finite.
\item The parallel toppling procedure for $s$ is finite.
\end{enumerate}
\end{corollary}

\begin{proof}
The implications $(1) \Rightarrow (2) \Rightarrow (3) \Rightarrow (4)$ and $(5) \Rightarrow (6)$ are immediate from the definitions.  The implication $(4) \Rightarrow (5)$ follows from Lemma~\ref{c.lap}.  Finally, let $u$ be the parallel toppling procedure for $s$. If $u$ is finite, then $u$ is both legal and stabilizing for $s$, which shows the remaining implication $(6) \Rightarrow (1)$.
\end{proof}


Denote by $\PP_o$ the law of the discrete time simple random walk  $(X_j)_{j \in \N}$ on $G$ started at $X_0 = o$.  Writing $p_j(x) = \PP_o(X_j = x)/ \deg(x)$, observe that the Laplacian \eqref{e.thelaplacian} of $p_j$ satisfies
 	\begin{align} \Delta p_j(x) &= \sum_{y \sim x} \left(\frac{\PP_o(X_j=y)}{\deg(y)} - \frac{\PP_o(X_j=x)}{\deg(x)} \right) \nonumber \\
	&=  \sum_{y \sim x} \PP_o(X_j=y, X_{j+1}=x) - \PP_o(X_j=x) \nonumber \\
	&= \deg(x) (p_{j+1}(x) - p_j(x)). \label{e.ptlaplacian}
	\end{align}	
\old{
 	\begin{align} \frac{\Delta p_t(x)}{\deg(x)} &= \frac{1}{\deg(x)} \sum_{y \sim x} (p_t(y) - p_t(x)) \nonumber \\
	&= \sum_{y \sim x} \PP_o(X_t=y, X_{t+1}=x) \nonumber \\
	&= p_{t+1}(x) - p_t(x). \label{e.ptlaplacian}
	\end{align}	
}
\old{
	 \begin{align*} \Delta p_t(x)
	&=  \deg(x) \left( -p_t(x) + \sum_{y \sim x} \PP_o(X_t=y, X_{t+1}=x) \right)  \\
	&= \deg(x) (p_{t+1}(x) - p_t(x)).
	\end{align*}
}
The next lemma will play an important role in the proof of the recurrent case of Theorem~\ref{main}. It relates the stabilizability of a particular configuration $s$ to the transience of the simple random walk.

\begin{lemma} \label{l.diracplusone}
Fix $o \in V$ and $\beta>0$. The divisible sandpile
	\[ s(x) = \begin{cases} 1, & x \neq o \\ 1+\beta, & x=o \end{cases} \]
stabilizes on $G$ if and only if the simple random walk on $G$ is transient.
\end{lemma}

\begin{proof}
 We compute the parallel toppling procedure $u_t$ and the configuration $s_t$ at time $t$ of Example \ref{e.parallel}.
Setting $p_t(x) = \PP_o(X_t=x)/\deg(x)$, let us show by induction on $t$ that
  	\begin{equation} \label{e.odomt} u_t(x)= \beta \sum_{j=0}^{t-1} p_j(x) \end{equation}
 and
 	\begin{equation} \label{e.sandt} s_t(x)= 1 + \beta \deg(x) p_t(x) \end{equation}
for all $x \in V$ and all $t \in \N$. Indeed, if \eqref{e.odomt} holds at time $t$, then summing \eqref{e.ptlaplacian} we obtain
	\[ s_{t} = s + \Delta u_{t} = s + \beta \deg(x) (p_t - p_0) = 1 + \beta \deg(x) p_t \]
so \eqref{e.sandt} holds at time $t$, whence
	\[ u_{t+1}(x) - u_t(x) = \frac{(s_t(x)-1)^+}{\deg(x)} = \beta p_t(x)  \]
so \eqref{e.odomt} holds at time $t+1$, completing the inductive step.

 Taking $t \uparrow \infty$ in \eqref{e.odomt} yields $u_\infty(x)= \beta g(o,x)/\deg(x)$ where
  	\begin{equation} \label{e.gdef} g(o,x)= \sum_{j=0}^{\infty} \PP_{o}(X_j=x) \end{equation}
is the Green function of $G$, which is finite if and only if $G$ is transient.

 If $G$ is transient, then the parallel toppling procedure $u$ is finite and $s + \Delta u_\infty \equiv 1$, so $s$ stabilizes. If $G$ is recurrent, then the parallel toppling procedure is infinite, so $s$ does not stabilize by Corollary~\ref{c.equiv}.
\end{proof}

\section{Conservation of density}
\label{s.conservation}

In this section we assume that $G =(V,E)$ is vertex-transitive, and we fix a subgroup $\Gamma$ of $\aut(G)$ that acts transitively: for any $x,y \in V$ there is an automorphism $\alpha \in \Gamma$ such that $\alpha x=y$.
For the rest of the paper, we assume $o \in V$ be an arbitrary fixed vertex.
Write $\mathcal{X}= \R^V$ (viewed as a measurable space with the Borel $\sigma$-field)
and $T_\alpha: \mathcal{X} \to \mathcal{X}$ for the shift $(T_\alpha f)(x)= f(\alpha^{-1} x)$.
Let $\PP$ be a probability measure on $\mathcal{X}$
satisfying $\EE |s(o)| <\infty$.  We assume that $\PP$ is \textbf{$\Gamma$-invariant}; that is, if $s$ has distribution $\PP$ then $T_\alpha s$ has distribution $\PP$ for all  $\alpha \in \Gamma$.


\begin{prop}
\moniker{Conservation of Density}
\label{p.cden}
If $\PP$ is $\Gamma$-invariant and $\PP \{ s \text{ stabilizes} \} = 1$, then the stabilization $s_\infty$ satisfies \[ \EE s_\infty(o) = \EE s (o). \]
\end{prop}

Fey, Meester and Redig \cite{FMR} used an ergodic theory argument to prove the conservation of density for the abelian sandpile on $\Z^d$ when $\Gamma$ is the group of translations.  They considered only nonnegative initial conditions: $s \geq 0$. We will give an elementary proof which starts in the same way, by choosing an automorphism-invariant toppling procedure (toppling in parallel) and uses the following observation about averages of uniformly integrable random variables.

\begin{lemma}
\label{l.unifint}
If $X$ is a random variable with $\EE |X|<\infty$, and $X_1, X_2, \ldots$ is any sequence of random variables such that $X_i \eqindist X$ for all $i$, then the family
	\[ \mathcal{S} = \{ a_1 X_1 + \ldots + a_k X_k \,|\, k \geq 1, a_i \geq 0, \sum a_i = 1 \} \]
is uniformly integrable.
\end{lemma}

\begin{proof}
We may assume $X \geq 0$. 
Given $\eps>0$, we must show that there is a $\delta>0$ so that for any set $A$ with $\PP(A)<\delta$ we have $\EE(Y \one_A)<\eps$ for all $Y \in \mathcal{S}$.

Since $\EE X < \infty$, there is such $\delta$ for X itself: choose $M$ so that $\EE(X \one_{\{X>M\}}) <\eps$ and then set $\delta=\PP(X>M)$.
Now for any set $A$ with $\PP(A)<\delta$ and any $Y = \sum a_i X_i \in \mathcal{S}$
	\[ \EE [Y \one_A] = \sum_{i=1}^k a_i \EE [X_i \one_A] < \eps \]
so the same $\delta$ works for all $Y \in \mathcal{S}$.
\end{proof}

\begin{proof}[Proof of Proposition~\ref{p.cden}]
We topple in parallel: at each time step $t=0,1,\ldots$, each site $x \in V$ distributes all of its excess mass $\sigma_t(x) = (s_t(x)-1)^+$ equally among its $r$ neighbors where $r$ is the common degree of all vertices in $G$.
The resulting configuration after $t$ time steps is	\[ s_t = s_0 + \Delta u_t. \]
where $u_t = r^{-1}(\sigma_0 + \ldots + \sigma_{t-1})$.  Since $\PP \{ s \text{ stabilizes} \}= 1$ we have $s_t(o) \to s_\infty(o)$, a.s..
We will show that the random variables $\sigma_t(o)$ for $t \in \N$ are uniformly integrable. To finish the proof from there, note that the law of $u_t$ is $\Gamma$-invariant, so $\EE u_t(x) = \EE u_t(y)$ for all $x,y \in V$. In particular, $\EE \Delta u_t(o)=0$ and hence $\EE s_t(o) = \EE s_0(o)$ for all $t<\infty$.
Since $s_t \geq \min(s_0,1)$ the uniform integrability of $\sigma_t(o)$ implies that of $s_t(o)$, so we conclude $\EE s_\infty(o)=\EE s_0(o)$.

At time step $t$ the origin retains mass $\leq 1$ and receives mass $\sigma_{t-1}(y)/r$ from each neighbor $y$, so
	\[ s_{t}(o) \leq 1 + \sum_{x \sim o} \frac{\sigma_{t-1}(x)}{r} \]
hence
	\[ \sigma_{t}(o) \leq \frac{1}{r} \sum_{x \sim o} \sigma_{t-1}(x). \]
Inducting on $t$ we find that
	\[ \sigma_t(o) \leq Y_t := \sum_{x \in V} a_t(x) \sigma_0(x) \]
where $a_t(x)$ is the probability that a $t$-step simple random walk started at  $o \in V$ ends at $x$.  By Lemma~\ref{l.unifint} the random variables $Y_t$ are uniformly integrable, which completes the proof.
\end{proof}

We remark that the above proof also applies to the abelian sandpile to show \cite[Lemma~2.10]{FMR}, by taking $\sigma_t(x) = \floor{s_t(x)^+/r}$, the number of times $x$ topples at time $t$.

\section{Behavior at critical density}

In this section $G=(V,E)$ is an infinite vertex-transitive graph, $o \in V$ is a fixed vertex, $\Gamma \subset \aut(G)$ is a group of automorphisms that acts transitively $V$, and $\PP$ is a $\Gamma$-invariant and ergodic probability measure on $\mathcal{X} =\R^V$ with $\EE |s(o)|<\infty$.
The event that $s$ stabilizes is $\Gamma$-invariant, so it has probability $0$ or $1$ by ergodicity.

\begin{lem} \label{l.notstab}
If $\EE s(o) >1$, then $\PP \{ s \text{ stabilizes} \} =0$.
\end{lem}

\begin{proof}
If $s$ stabilizes, then by conservation of density (Proposition~\ref{p.cden}),  we have $\EE s_\infty(o) =\EE s(o)$.
 Since the final configuration $s_\infty$ is stable we have $s_\infty(o) \leq 1$, so $\EE s(o) \leq 1$.
 \end{proof}

\begin{lem} \label{l.stable}
If $\EE s(o) <1$ then $\PP \{ s \text{ stabilizes} \} =1$.
\end{lem}

\begin{proof}
We will show the contrapositive. We will use the observation made in the introduction that if $s_t(o) \geq 1$ for some time $t$ then $s_T(o) \geq 1$ for all $T \geq t$.

Consider first the case that $s$ is bounded below: $\PP(s(o) \ge M)=1$ for some $M \in (-\infty,0]$.
Define $u_t$ and $s_t$ by toppling in parallel as in \textsection\ref{s.conservation}.
Supposing that $\PP \{s \text{ stabilizes}\} =0$, we have $u_t(o)>0$ for some sufficiently large $t$, a.s.; this fact is  contained in the proof of Lemma~\ref{c.lap}.
Since $u_t(o)>0$ implies $s_t(o) \geq 1$ and hence $s_T(o) \geq 1$ for all $T \geq t$, we have
	\[ \PP \left\{ \liminf_{t \to \infty} s_t(o) \ge 1 \right \} = 1. \]
Since $s_t \ge \min(s,1) \ge M$, by Fatou's lemma $\EE \left(\liminf_{t\to \infty} s_t(o) \right) \le \EE s(o)$, which shows $\EE s(o) \geq 1$ as desired.

Now we use a truncation argument to reduce the general case to case where $s$ is bounded below. Choose sufficiently small $M \in (-\infty,0]$ such that
$\EE s(o) \one_{s(o) \ge M} < 1$. By the previous case the configuration $s \one_{s \ge M}$ stabilizes almost surely. Since $s \le s \one_{s \ge M}$, we have that
$s$ stabilizes almost surely.
\end{proof}

If $s$ stabilizes then the \textbf{odometer} of $s$ is the function $u_\infty : V \to [0,\infty)$, where $u$ is any legal stabilizing toppling procedure for $s$.

\begin{lem}
\label{l.allones}
If $s$ has $\Gamma$-invariant law $\PP$ and $\PP \{s \text{ stabilizes} \}=1$, then the odometer $u_\infty$ has $\Gamma$-invariant law. Moreover, if $\EE s(o) = 1$ then $s_\infty \equiv 1$ and $\Delta u_\infty = 1 - s$.
\end{lem}

\begin{proof}
We use \eqref{e.LAP} along with the fact that $\Delta$ commutes with $T_\alpha$: if $s + \Delta f \leq 1$ then $T_\alpha s + \Delta (T_\alpha f) = T_\alpha (s + \Delta f) \leq 1$, so if $u_\infty$ is the odometer for $s$ then $T_\alpha u_\infty$ is the odometer for $T_\alpha s$.

By conservation of density (Proposition~\ref{p.cden}), $\EE s_\infty(o)=1$. Since $s_\infty \leq 1$ it follows that $s_\infty \equiv 1$, and hence $\Delta u_\infty = 1-s$.
\end{proof}

In the case that $G$ is recurrent we can prove our main theorem with no moment assumption, using an ``extra head'' construction.

\begin{lem} \label{l.recurrent}
Suppose that $G$ is an infinite, recurrent, vertex-transitive graph and $\{ s(x)\}_{x \in V}$ are i.i.d.\ with $\mathbb{E} s(x) = 1$  and $\mathbb{P} \{s(x)=1 \} \neq 1$. Then \[ \PP \{ s \mbox{ stabilizes} \} =0.\]
\end{lem}

\begin{proof}
Let $S \subset \mathcal{X}$ denote the set of configurations that stabilize.
 By ergodicity, $\mathbb{P}(S) \in \{0,1\}$.
We will show that if $\PP(S)=1$, then the graph $G$ must be transient. Let $s:V \to \R$ denote an i.i.d.\  configuration with the given distribution such that $\EE s=1$, and fix a vertex $o \in V$.
 We create a new i.i.d.\ configuration $s'$ with the same law as $s$, by independently resampling $s(o)$ from the same distribution.
Then $s'= s + \beta \delta_o$, where $\delta_o(x) = \one \{x = o\}$ and $\beta$ is a mean zero random variable. Since $\Var s >0$, we have that $\PP(\beta >0) >0$. Using $\PP(S ) =1$ along with Lemma~\ref{l.allones}, there exist
 $s:V \to \R$ and $\beta>0$ such that $s$ and $s+\beta \delta_o$ both stabilize to the all $1$ configuration.
By toppling $s+\beta \delta_o$ in two stages (Example~\ref{e.stages}), it follows that $1+\beta \delta_o$ stabilizes to $1$, so $G$ is transient by Lemma \ref{l.diracplusone}.
\end{proof}

In the preceding lemma the hypothesis that $s$ is i.i.d.\ can be substantially weakened: The proof uses only the fact that with positive probability, the conditional distribution of $s(o)$ given $\{s(x)\}_{x \neq o}$ is not a single atom.

\section{Proof of Theorem \ref{main}}

\subsection{Singly transient case} \label{s.strans}

Recall Green's function \eqref{e.gdef}. In this section we assume that $g(o,y)<\infty$ for all $y \in V$ but
	\begin{equation} \label{e.singlytransient} \sum_{y \in V} g(o,y)^2 = \infty. \end{equation}
Define $V_n= \{ x \in V : d(x,o) \le n \}$ where $d$ is the graph distance. 
Then $V_1 \subset V_2 \subset \ldots$ are finite sets with $\bigcup_{n \geq 1} V_n = V$.
Let $g_n(x,y)$ be the expected number of visits to $y$ by simple random walk started at $x$ and killed on exiting $V_n$.
By the monotone convergence theorem, for fixed $x,y \in V$ we have
	\[ g_n(x,y) \uparrow g(x,y) \]
as $n \to \infty$.
In particular, setting
	\[ \nu_n := \left( \sum_{y \in V_n} g_n(o,x)^2 \right)^{1/2} \]
we have $\nu_n \uparrow \infty$ as $n \to \infty$ by \eqref{e.singlytransient}.

The proof of the following lemma is inspired by
	\cite[Theorem 3.1]{FR} and \cite[Theorem 3.5]{FMR}.
\begin{lemma} \label{l.nondeg} Let $G=(V,E)$ be singly transient \eqref{e.singlytransient},  vertex transitive graph.
Let $\{ s(x)\}_{x \in V}$ be i.i.d.\ with $\mathbb{E} s = 1$ and $\Var s < \infty$. If \[ \frac{1}{\nu_n} \sum_{x \in V_n} g_n(o,x) (s(x) - 1) \]
converges in distribution as $n \to \infty$ to a nondegenerate normal random variable~$Z$, then $\PP \{ s \text{\emph{  stabilizes}} \} = 0$.\end{lemma}

\begin{proof}
Assume to the contrary that $s$ stabilizes a.s.. Let $u$ be the nested volume toppling procedure (Example \ref{e.nest}). Then for each $n \in \N$,
\begin{equation}\label{e.finitevolume}
s + \Delta u_n = \xi_n \qquad \text{on } V_n
\end{equation}
with $\xi_n \leq 1$ on $V_n$.
Equation \eqref{e.finitevolume} can be rewritten as
\[
u_n (y) = r^{-1}\sum_{x \in V_n} g_n(x,y) (s(x) -1) + r^{-1}\sum_{x \in V_n} g_n(x,y) (1  - \xi_n(x))
\]
where $r$ is the common degree
(both sides have Laplacian $\xi_n - s$ in $V_n$ and vanish on $V_n^c$).
Observe that the second term is a nonnegative random variable. Therefore, for any $\eps>0$
\begin{equation} \label{e.decomp}
\mathbb{P} \{ u_n (o) > \epsilon \nu_n \}  \geq \mathbb{P} \left\{ \frac{1}{\nu_n} \sum_{x \in V_n} g_n(o,x) (s(x) - 1) >r \epsilon \right\}.
\end{equation}
Since $u$ is a legal toppling procedure and we have assumed that $s$ stabilizes a.s., we have $u_\infty(o) < \infty$, a.s..
Now since
$u_n(o) \uparrow u_\infty(o)$ and $\nu_n \uparrow \infty$, the left side of \eqref{e.decomp} tends to zero as $n \rightarrow \infty$.
However, the right side tends to a positive limit $P(Z>r\epsilon)>0$, which gives the desired contradiction.
\end{proof}

To complete the proof of Theorem~\ref{main} in the singly transient case, it remains to show that $\frac{1}{\nu_n} \sum_{x \in V_n} g_n(o,x) (s(x) - 1)$ converges in distribution to a nondegenerate normal random variable. We show this using Lindeberg-Feller central limit theorem as follows.

 \begin{lemma}\label{l.clt}
  Let  $\{X_{n1}: n \ge 1; i=1,\ldots,k_n\}$ be a triangular array of identically distributed random variables such that for each $n \in \N$,  $\{X_{ni}: i=1,\ldots,k_n\}$ is independent.
  Assume that $E[X_{11}]=0$ and $E[X_{11}^2]=1$. Let $a_{nk}>0$ be such that $\sum_{k=1}^{k_n} a_{nk}^2= 1$ and $\lim_{n \to \infty} b_n=0$ where $b_n = \max_{1\le k \le k_n} a_{nk}$. Then the sequence
  $Y_n= \sum_{k=1}^{k_n} a_{nk}X_{nk}$ converges in distribution to standard normal random variable as $n \to \infty$.
 \end{lemma}
\begin{proof}
 By Lindeberg-Feller central limit theorem \cite[Theorem 27.2] {Bil} it suffices to check the Lindeberg condition:
 \[
  \lim_{n \to \infty} \sum_{k=1}^{k_n} \EE[ a_{nk}^2 X_{nk}^2 \mathbf{1}_{\abs{a_{nk} X_{nk}} > \epsilon} ]= 0
 \]
for all $\epsilon>0$. Since
$ \EE[ X_{nk}^2 1_{\abs{a_{nk} X_{nk}} > \epsilon} ] \le  \EE[ X_{11}^2 1_{\abs{ X_{11}} > \epsilon/b_n} ] $ and $\sum_k a_{nk}^2=1$,
we have
 \[
  \lim_{n \to \infty} \sum_{k=1}^{k_n} \EE[ a_{nk}^2 X_{nk}^2 \mathbf{1}_{\abs{a_{nk} X_{nk}} > \epsilon} ] \le \lim_{n \to \infty} \EE[ X_{11}^2 \mathbf{1}_{\abs{ X_{11}} > \epsilon/b_n} ] =0
 \]
 because  $\lim_{n \to \infty} b_n=0$ and $\EE[X_{11}^2]=1$.
\end{proof}

\begin{lemma}\label{l.singlytransient}
Let $G=(V,E)$ be a singly transient \eqref{e.singlytransient}, vertex transitive graph. Suppose $\{ s(x)\}_{x \in V}$ are i.i.d.\ random variables with $\EE s = 1$ and $ 0< \var s < \infty$. Then $\PP \{s \text{ stabilizes} \} = 0$.
\end{lemma}
\begin{proof}
Since $g(o,o)< \infty$ and $\nu_n \uparrow \infty$, we have
\[
 \lim_{n \to \infty} \max_{x \in V_n} \frac{g_n(o,x)}{\nu_n}  =  \lim_{n \to \infty}  \frac{g_n(o,o)}{\nu_n} \le   \lim_{n \to \infty}  \frac{g(o,o)}{\nu_n} =0.
\]
For $n \in \N$ and $x \in V_n$, define $X_{nx}= (s(x)-1)/\sqrt{\var{s}}$ and $a_{nx}= \frac{g_n(o,x)}{ \nu_n}$ and
$b_{n}= \max_{x \in V_n} a_{nx}$. By Lemma~\ref{l.clt}, we have that
\[ \frac{1}{\nu_n} \sum_{x \in V_n} g_n(o,x) (s(x) - 1) \] converges in distribution to a normal random variable with mean 0 and variance $\var(s)$. Lemma~\ref{l.nondeg} implies the desired conclusion.
\end{proof}

\subsection{Doubly transient case} \label{s.dtrans}

We start by outlining our proof strategy. Recall from Lemma~\ref{l.allones} that if $s$ stabilizes with $\EE s(o)=1$ then it must stabilize to the constant configuration $s_\infty \equiv 1$. In particular, the odometer  $u_\infty$ satisfies $\Delta u_\infty = 1-s$.  In Lemma~\ref{l.vfun} below, by convolving $s-1$ with Green's function we can build an explicit function $v$ with Laplacian $1-s$; the convolution is defined almost surely provided that
	\begin{equation} \label{e.greensumofsquares} \sum_{y \in V} g(o,y)^2 < \infty. \end{equation}
(This condition says that the expected number of collisions of two independent random walks started at $o$ is finite.  The essential feature of Green's function here is of course that
	\begin{equation} \label{e.laplacianofg} \Delta g(\cdot,y) = -r\delta_y \end{equation}
where $r$ is the common degree of all vertices of $G$, and $\delta_y(x) = \one\{x = y\}$.)

Having built the function $v$, the difference $v-u_\infty$ is then a random harmonic function with $\Gamma$-invariant law, which must be an almost sure constant by the following lemma.

\begin{lemma}
\label{l.stationaryharmonic}
\moniker{Harmonic Functions With Invariant Law}
Let $\Gamma$ be a group of automorphisms of $G$ that acts transitively on the vertex set $V$.
Suppose that $h: V \to \R$ has $\Gamma$-invariant law and $\Delta h \equiv 0$.
If $h$ can be expressed as a difference of two functions $h=v-u$ where $u \geq 0$ and $\sup_{x \in V} \EE v(x)^+ <\infty$, then $h$ is almost surely constant.
\end{lemma}

\begin{proof}
Let $(X_n)_{n \geq 0}$ be simple random walk on $G$ started at $X_0=o$. Although $h$ is harmonic, $h(X_n)$ need not be a martingale since it need not have finite expectation. But for any $a \in \R$, since $h = v-u$ and $u \geq 0$, the truncation
	\begin{align*} M_n &:= a + (h(X_n)-a)^+ \\
		&\ghost{:}\leq a + (v(X_n) - a)^+ \end{align*}
has finite expectation. Since the function $t \mapsto (t-a)^+$ is convex, $x \mapsto (h(x)-a)^+$ is subharmonic, so $M_n$ is a submartingale:
	\[ \EE[M_{n+1} | h, X_0, \ldots, X_n] = \frac1r \sum_{w \sim X_n} (a + (h(w)-a)^+) \geq M_n. \]
A submartingale bounded in $L^1$ converges almost surely \cite[11.5]{Williams}. Since this holds for any $a \in \R$ it follows that $h(X_n)$ converges almost surely.
But since $h$ has $\Gamma$-invariant law, $h(X_n)$ is a stationary sequence,
so the only way it can converge a.s.\ is if $h(X_n)=h(X_0)$ for all $n$. Since $G$ is connected, for any vertex $x \in V$ there exists $n$ with $\PP(X_n=x)>0$, so $h(x)=h(o)$.
\old{
 Thus the usual proof of upcrossing lemma works. Since $u \geq 0$ we have $\sup \mathbb{E}h(X_n)^+ \le \mathbb{E} v(X_n)^+ < \infty$, and
we can modify to the proof of martingale convergence theorem (Theorem 5.2.8 in \cite{Dur}) to show that $h(X_n)$ converges almost surely.
}
\end{proof}

\begin{lem} \label{l.vfun} Let $\Gamma$ be a countable subgroup of $\aut(G)$ which acts transitively on the vertices of $V$.
Suppose $\{ \sigma(x) \}_{x \in V}$ are i.i.d.\ random variables with $\mathbb{E} \sigma(x) = 0$ and $0 < \var \sigma(\zero) < \infty$ Let $y_1,y_2,\ldots$ be an enumeration of the vertex set of $G$. For $\alpha \in \Gamma$, let
	\begin{equation} \label{e.vz} v_\alpha(x) := \frac{1}{r} \sum_{i=1}^\infty g(x,\alpha y_i) \sigma(\alpha y_i). \end{equation}
If $G$ is doubly transient \eqref{e.greensumofsquares}, then the following hold almost surely.
	\begin{enumerate}[{\em (a)}]
	\item The series defining $v_\alpha(x)$ converges for all $x \in V, \alpha \in \Gamma$.
	\item $v_\alpha = v_\beta$ \ for all $\alpha, \beta \in \Gamma$.
	\item $\Delta v_e = -\sigma$ where $e$ denotes the identity automorphism.
	\item $v_e$ has $\Gamma$-invariant law, that is $T_\alpha v_e \eqindist v_e$ for all $\alpha \in \Gamma$.
	\item $v_e$ is unbounded above and below.
	\end{enumerate}
\end{lem}

\begin{proof}
(a) Each term in the series \eqref{e.vz} is independent, and the $i$-th term has variance $g(x,\alpha y_i)^2 U$ where $U = r^{-2}\EE \sigma(\zero)^2$.  By \eqref{e.greensumofsquares} the sum of these variances is finite, which implies that the series converges a.s..

(b) Note that there is something to check here because the series defining $v_\alpha$ is only conditionally convergent.
Fix $x \in V$ and denote the partial sum $r^{-1}\sum_{i=1}^n  g(x,\alpha y_i) \sigma(\alpha y_i)$ by $v_{\alpha,n}$.
We compare $v_{\alpha,n}$ and $v_{e,n}$ where $e$ is the identity element of $\Gamma$.
Given $\epsilon > 0$, choose $n_1$ such that
	\[ U \sum_{i=n_1}^\infty  g(x, y_i)^2 < \epsilon^3. \]
 There exists $N_1$ depending on $n_1$ and $\alpha$ such that
 	\[ \{ y_i: i=1,2,\ldots,n_1 \} \subset \{ \alpha y_i: i=1,2,\ldots,N_1 \}. \]
This implies that for any $N \ge N_1$ we have $\var \left( v_{\alpha,N}-  v_{e,N} \right)
< 2\epsilon^3$. By Chebyshev's inequality,
\begin{eqnarray*}
\mathbb{P} \left( \left| v_{\alpha,N}-  v_{e,N} \right| >  \epsilon \right) \le \eps^{-2} \var \left( v_{\alpha,N} -  v_{e,N} \right)
 < 2 \epsilon
\end{eqnarray*}
for all $N \ge N_1$.
By part (a) there a.s.\ exists $N_2 \ge N_1$
such that $\max(|v_{\alpha,n}- v_\alpha(x)|,|v_{e,n}- v_e(x)|)< \eps/3$ for all $n \ge N_2$. By the triangle inequality it follows that
\begin{eqnarray*}
\mathbb{P} \left( \left| v_{\alpha}(x)-  v_{e}(x)\right| >  \frac{\eps}{3} \right) <  2 \epsilon.
\end{eqnarray*}
Since $\eps$ was arbitrary we obtain $v_\alpha(x) = v_e(x)$, a.s.. By taking countable intersections we have that $v_\alpha = v_e$, a.s..

(c) Using \eqref{e.laplacianofg}, we compute
\begin{eqnarray*}
\Delta v_e (x) &=& \frac{1}{r} \sum_{w\sim x} (v_e(w) - v_e(x)) \\
 &=& \frac{1}{r} \sum_{w \sim x} \left( \sum_{i=1}^\infty g(w,y_i) \sigma(y_i) - \sum_{i=1}^\infty  g(x,y_i) \sigma(y_i) \right) \\
  &=&  \sum_{i=1}^\infty \sigma(y_i) \frac{1}{r} \sum_{w \sim x} \left(  g(w,y_i)  -   g(x,y_i)  \right) \\
&=&  - \sum_{i=1}^\infty \sigma(y_i) \one \{y_i =x \} \\
&=& - \sigma(x).
\end{eqnarray*}

(d) To show that $v_e$ has $\Gamma$-invariant law, write $v_e = v_e^\sigma$ to make the dependence on the initial configuration $\sigma$ explicit. We have for all $\alpha \in \Gamma$
	\begin{align*}
	T_\alpha v_e^\sigma (x)
		&= \sum_{i=1}^\infty g(\alpha^{-1}x, y_i) \sigma(y_i) \\
		&= \sum_{i=1}^\infty g(x, \alpha y_i) \sigma(y_i) \\
		&= \sum_{i=1}^\infty g(x, \alpha y_i) (T_\alpha \sigma)(\alpha y_i) \\
		&=  v_\alpha^{T_\alpha \sigma}(x) \\
		&= v_e^{T_\alpha \sigma}(x)
	\end{align*}
where in the last equality we have used part (b). Hence $T_\alpha v_e^\sigma = v_e^{T_\alpha \sigma} \eqindist v_e^\sigma$ since $\sigma$ has $\Gamma$-invariant law.

(e)
To show that $v_e$ is almost surely unbounded below, we use the assumption that $\sigma(o)$ has zero mean and positive variance, which implies that
	\[ \PP(\sigma(o) < -\delta)>p \]
for some $p,\delta>0$.  Since $\sum_{y \in V} g(o,y) = \infty$ and $\sum_{y \in V} g(o,y)^2<\infty$, we can choose $N$ large enough so that
	\[ \delta r^{-1}\sum_{i=1}^N g(o,y_i) > 2M, \qquad U \sum_{i=N+1}^\infty g(o,y_i)^2 < 1. \]	
By Chebyshev's inequality
	\[ \PP ( r^{-1}\sum_{i=N+1}^\infty g(x,y_i) \sigma(y_i) \geq M) \leq \frac{1}{M^2}. \]
On the event that $\sigma(y_i)<-\delta$ for all $1 \leq i \leq N$ and $r^{-1}\sum_{i=N+1}^\infty g(x,y_i) \sigma(y_i) \geq M$ we have $v_e(o) < -2M + M$. Since the random variables $\sigma(y_i)$ are i.i.d., we obtain
	\[ \PP ( v_e(o) < -M) \geq p^N (1 - \frac{1}{M^2}) > 0. \]
Since $v_e$ is stationary, we have $\PP( \inf v_e < -M ) \in \{0,1\}$ by ergodicity. Since this probability is $>0$, it must be $1$. Since $M$ was arbitrary, $v_e$ is a.s.\ unbounded below.

A similar argument shows that $v_e$ is also a.s.\ unbounded above.
\end{proof}

\begin{lem}
Let $G=(V,E)$ be a doubly transient \eqref{e.greensumofsquares}, vertex transitive graph and let $\{ s(x)\}_{x \in V}$ be i.i.d.\ random variables with $\EE s = 1$ and $0 < \Var s < \infty$. Then $\PP \{ $s$ \text{ stabilizes} \} = 0$.
\end{lem}

\begin{proof} Assume to the contrary that $s$ stabilizes a.s.\ with odometer $u_\infty$.
Let $\Gamma$ be a countable subgroup of $\aut(G)$ that acts transitively on the vertices of $G$.
Let $v : V \to \R$ be given by equation \eqref{e.vz} with $\alpha = e$ and $\sigma = s-1$. By Lemma \ref{l.vfun}, $v$ has $\Gamma$-invariant law and $\Delta v = 1-s$. Therefore
	\[ h := v-u_\infty \]
has $\Gamma$-invariant law and $h$ is harmonic: $\Delta h \equiv 0$ on $V$.
Further, by Fatou's lemma, $\mathbb{E} v(x)^2 \leq \deg(x)^{-2} \var (s) \sum_{i=1}^\infty g(x,y_i)^2 < \infty$. Lemma~\ref{l.stationaryharmonic} now implies that $h$ is almost surely constant.
This contradicts Lemma \ref{l.vfun} because $u_\infty \ge 0$ and $v$ is almost surely unbounded below.
\end{proof}

\section{Stabilizability of Cones}
\label{s.cones}

Until now we have mainly been concerned with the stabilizability of random initial configurations. In this section we examine stabilizability of a few deterministic configurations on the square grid $\Z^2$.  We present two examples, one of which stabilizes.

\begin{lem}
Define $C_1 = \{ (x,y) \in \mathbb{Z}^2 : x \ge 0, \abs{y} \le x \}$. Then the configuration
$s_0= (1+ \alpha) \one_{C_1}$ does not stabilize for $\alpha>0$.
\end{lem}

\begin{proof}
By least action principle (Proposition \ref{lap}), it suffices to show the existence of an infinite legal toppling procedure.

Let $u_k$ denote the parallel toppling procedure of Example \ref{e.parallel} where $k \in \N$.
Let $C= \{ (x,y)\in \mathbb{Z}^2: x > 0, \abs{y}< x  \}$. Let $(X_n,Y_n)$ denote the simple random walk on $\mathbb{Z}^2$ and let $N$ denote the stopping time $N= \min \{ n \ge 0: X_n= \abs{Y_n} \}$.
As in the proof of Lemma \ref{l.diracplusone}, we keep track of the mass from each $(x,y) \in C$ to obtain
\[
 u_k(1,0) \ge \frac{\alpha}{4} \sum_{(x,y) \in C} \sum_{l=0}^k \PP_{(x,y)} ( (X_{N},Y_{ N}) = (0,0), N=l  ).
\]
for all $k \in \N$.
To see this note that $N$ is the exit time of the set $C$ and the only way to exit $C$ at $(0,0)$ is from $(1,0)$.
As a result, we have
\[ u_\infty(1,0) \ge \frac{\alpha}{4} \sum_{(x,y) \in C} p(x,y) \]
 where $p(x,y)=\mathbb{P}^{(x,y)}\left( (X_N,Y_N)=(0,0) \right)$.
Consider the function $q:\mathbb{Z}^2 \to [0,1]$ given by
\begin{equation*}
q(x,y)=
\begin{cases} p(\abs{x},y) & \text{if $\abs{y}< \abs{x}$,}
\\
-p(\abs{y},x) &\text{if $\abs{x} <  \abs{y}$,}
\\
0 & \text{ otherwise.}
\end{cases}
\end{equation*}
Note that $\Delta q =  \delta_{(0,1)} + \delta_{(0,-1)} -\delta_{(1,0)}  -\delta_{(-1,0)}$. Let $g(x,y)$ denote the potential kernel in $\mathbb{Z}^2$ defined by
\[
 g(x,y)= \sum_{n=0}^\infty [\PP_{(0,0)} (X_n = (x,y)) - \PP_{(0,0)}(X_n = (0,0))]
\]
where $(X_n)_{n \in \N}$ denotes the simple random walk on $\Z^2$. Although the simple random walk on $\Z^2$ is transient, it turns out that the sum defining
$g$ is absolutely convergent.
By standard estimates on $g$ (See \cite[Chapter 1]{Law}), we know that $g$ has sub-linear (logarithmic) growth.
This combined with  \cite[Theorem 6.1]{HS} implies that
	\[ q(x,y) - \frac{1}{4} \left( g(x+1,y)+g(x-1,y)-g(x,y-1)-g(x,y+1)  \right) \] is identically zero because it is harmonic with  sub-linear growth and attains the value $0$ at $(0,0)$.
Therefore there exists $c_1>0$ such that for all $(x,y) \in C$, we have
\begin{eqnarray*}
q(x,y)&=&  \frac{1}{4} \left( g(x+1,y)+g(x-1,y)-g(x,y-1)-g(x,y+1)  \right) + \frac{1}{4} \Delta g(x,y)\\\
&=& \frac{1}{2} \left( g(x+1,y)+g(x-1,y)-2g(x,y) \right) \\
& > & c_1 \frac{x^2 - y^2}{(x^2+y^2)^2}
\end{eqnarray*}
The first line above follows from the fact that $\Delta g = 0$ for all points except $(0,0)$.
The last line above follows from \cite[Theorem 1.6.5 (b)]{Law}.
Therefore
$ u_\infty(1,0) \ge \frac{\alpha}{4} \sum_{(x,y) \in C} q(x,y) > \frac{\alpha c_1}{4} \sum_{(x,y) \in B} \frac{x^2-y^2}{ (x^2 +y^2)^2} = \infty$. Hence $s_0$ does not stabilize.
\end{proof}

We need the following technical lemma for the next example:
 \begin{lemma} \label{l.hillhole}
 Let $\sigma: V \to \R$ be a configuration in $G=(V,E)$. Assume that $H:= \{x : \sigma(x) > 1 \}$ satisfies $\abs{ H } < \infty $ and
 \[ \sum_{x \in H} (\sigma(x)- 1)^+ < \infty.
 \]
Let $F \subset V$ be such that $\abs{F} < \infty$ and
\[
   \sum_{x \in V} (\sigma(x)- 1)^+  \le   \sum_{x \in F} (1-\sigma(x))^+.
\]
Then $\sigma$ stabilizes.
\end{lemma}
\begin{proof}
 Let $a: H \times F \to [0,\infty)$ be a non-negative function such that
 $\sum_{y \in F} a(x,y) = \sigma(x) -1  $ for all $x \in H$ and
 $ \sum_{x \in H} a(x,y) \le 1 -\sigma(y)$
for all $y \in F$.  The function $a$ encodes how to redistribute the excess mass from $H$ to $F$.

Let \[ f_{x,z}(y) = \frac{g_z(x,y)}{\deg(y)} = \frac{ \EE^x ( \mbox{number of vists to $y$ before being killed at $z$})}{\deg(y)}\] denote the Green's function normalized with degree.
Observe that $\Delta f_{x,z} = \delta_z - \delta_x$.
Therefore the function $u:= \sum_{x \in H, y \in F} a(x,y) f_{ x, y}$ satisfies $\sigma  + \Delta u \le 1$ and $u \ge 0$.
This in turn implies that $\sigma$ stabilizes.
\end{proof}
\begin{remark}
 The condition that $\abs{H}, \abs{F} < \infty$ in the above lemma is necessary. The following example illustrates this: Consider a probability measure $\mu$ on $\N^* = \{ 1,2,3,\ldots\}$
 and consider the function $\sigma_\mu = 1 + \delta_0 - \mu(x)$ where $\mu(x) = \mu ( \{ x\})$. Then it can be shown that $\sigma_\mu$ stabilizes if and only if $\sum_{x \in \N^*} x \mu(x) < \infty$.
\end{remark}

\begin{lem}
\label{l.conestab}
 Define  $C_a= \{ (x,y) \in \mathbb{Z}^2: x \ge 0, \abs{y} \le ax \}$.
 Then the configuration\ $s_a= m \one_{C_a}$ stabilizes if $\frac{2ma}{1+a^2} \le 1$ and $a \in (0,1]$. Moreover $s_0(x,y) = x \one_{\{x > 0, \, y=0\}}  $ stabilizes.
\end{lem}
\begin{proof}
  The case $a=1$ is trivial. Define for $a \in (0,1]$  \[u_a(x,y) = \frac{ \left(ax - \abs{y}\right)^2 }{2(1+a^2) } \one_{C_a} .\]

To see that $s_0$ stabilizes, we check that $s_0 + \Delta u_1 \le 1$. This follows immediately from the computation of $\Delta u_1$ as
\begin{eqnarray*}
\Delta u_1(x,y)&=& \begin{cases} 1-x & \mbox{if } x > 0 ,y =0 \\
1 & \mbox{if }\abs{y} < x, x>0  \\
\frac{1}{2}& \mbox{if }\abs{y} = x, x>0  \\
\frac{1}{4} & \mbox{if } x=y=0  \\
0 & \mbox{otherwise.}
\end{cases}
\end{eqnarray*}

\old{  In $C_a$ we have
 \begin{eqnarray*}
(1+a^2) (u_1 - m u_a) &=&  (1+a^2)(x- \abs{y})^2 - 2m (ax - \abs{y})^2 \\
& \ge & 2m\left( a(x- \abs{y})^2 -(ax - \abs{y})^2 \right)\\
& = & 2m(1-a) \left( a x^2 - \abs{y}^2 \right)\\
& \ge & 2m(1-a) \left( a^2 x^2 - \abs{y}^2 \right) \\
 & \ge 0
\end{eqnarray*}
Outside $C_a$ it is clear that $u_1 - m u_a \ge 0$. }
A direct computation yields $\Delta u_a$ in different regions:
$\Delta u_a(0,0)= \frac{a^2}{2(1+a^2)} \le \frac{1}{4}$.
If $y=0, x > 0$ and all neighbors of $(x,0)$ are in $C_a$ (\emph{i.e.} $x \ge \lceil 1/a \rceil$), then
\begin{eqnarray*}
 \Delta u_a(x,0) &=& 1- \frac{2a}{1+a^2} x.
\end{eqnarray*}
If $y=0, x > 0$ and $(x,\pm 1) \notin C_a$, then
\begin{eqnarray*}
 \Delta u_a(x,0) &=& \frac{a^2(1-x^2)}{1+a^2}.
\end{eqnarray*}
If all neighbors of $(x,y)$ are in $C_a$ with $y \neq 0$, then
\begin{eqnarray*}
 \Delta u_a(x,y) &=& 1.
\end{eqnarray*}
If $(x,y) \in C_a$ with $y \neq 0$ and one of the neighbors is not in $C_a$, then
\[
0 < \frac{a^2}{1+a^2} \le \Delta u_a(x,y)  \le 1.
\]
If $(x,y) \notin C_a$ and if all neighbors of $(x,y)$ are not in $C_a$, then
\begin{eqnarray*}
 \Delta u_a(x,y) &=& 0.
\end{eqnarray*}
If $(x,y) \notin C_a$ and one of the neighbors in $C_a$, then
at-least 2 of the neighbors are not in $C_a$ and
\[
0 \le \Delta u_a(x,y)   < 1.
\]
Consider $v_a(x,y) = u_1(x+ \lceil 1/a \rceil,y) - m u_a(x+ \lceil 1/a \rceil,y) $. It is easy to check that
 $v_a \ge  0$ if  $\frac{2ma}{1+a^2} \le 1$.

If $x>0$ we have
\begin{eqnarray*}
s_a(x,0)+ \Delta v_a(x,0) &=& m  + \left(1- \lceil 1/a \rceil - x\right) -  m \left(1 - a\left( x+ \lceil 1/a \rceil\right) \right) \\
&\le & 1+ \left( \frac{2am}{1+a^2} - 1 \right)  \left(x+ \lceil 1/a \rceil\right) \\
& \le & 1.
\end{eqnarray*}
If $x>0,y\neq 0, (x,y) \in C_a$, then
\begin{eqnarray*}
s_a(x,y)+ \Delta v_a(x,y) &=& m  + 1 - m = 1.
\end{eqnarray*}
If $y\neq 0, (x,y) \notin C_a$, then
\begin{eqnarray*}
s_a(x,y)+ \Delta v_a(x,y) &\le& \Delta u_1 \le 1.
\end{eqnarray*}
If $x< - \lceil 1/a \rceil - 1$, then
\begin{eqnarray*}
s_a(x,y)+ \Delta v_a(x,y) &=& 0.
\end{eqnarray*}
Therefore $s_a+ \Delta v_a \le 1$ for all points except on the finite set $\{ (x , 0) : - \lceil 1/a \rceil - 1 \le x \le 0 \}$ and
$s_a+ \Delta v_a = 0$ for all $x <-1  - \lceil 1/a \rceil$. Lemma \ref{l.hillhole} implies that $s_a+ \Delta v_a $ stabilizes, and therefore
$s_a$ stabilizes.
\end{proof}

We conjecture that the bound in Lemma~\ref{l.conestab} is sharp.

\begin{conjecture}
 Define for $a \in (0,1]$, $C_a= \{ (x,y) \in \mathbb{Z}^2: x \ge 0, \abs{y} \le ax \}$.
 Then the divisible sandpile\ $s_a= m \one_{C_a}$ stabilizes if and only if $\frac{2ma}{1+a^2} \le 1$. Furthermore, the divisible sandpile $s_0 = k x \one_{\{x>0,\, y=0\}}  $ stabilizes if and only if $k\le 1$.
\end{conjecture}

More generally, we have the following problem.
\begin{open}[Tests for stabilizability] Given $s:\Z^d \to \R$, find series tests or other criteria that can distinguish between stabilizing and exploding $s$.
\end{open}

\section{Finite graphs}

Let $G= (V, E)$ be a finite connected graph with $|V| = n$. For a finite connected graph, all harmonic functions are constant: the kernel of $\Delta$ is $1$-dimensional spanned by the constant function $1$.

\begin{lemma}
\label{l.massn}
Let $s: V \to \R$ be a divisible sandpile with $\sum_{x \in V} s(x) = n$. Then $s$ stabilizes to the all $1$ configuration, and the odometer of $s$ is the unique function $u$ satisfying $s + \Delta u = 1$ and $\min u =0$.
\end{lemma}

\begin{proof}
Since $\Delta$ has rank $n-1$ and $\sum_{x \in V} (s(x) -1) = 0$, we have $s-1 = \Delta v$ for some $v$. Letting $w = v-\min v$, we have $w \geq 0$ and $s + \Delta w = 1$, so $s$ stabilizes.

Now if $u$ is any function satisfying $s+ \Delta u \leq 1$, then
	\[ \sum_{x\in V} (s+\Delta u)(x) = n \]
so in fact $s+\Delta u = 1$.  This shows that $s$ stabilizes to the all $1$ configuration, and moreover any two functions $u$ satisfying $s+\Delta u \leq 1$ differ by an additive constant. By the least action principle (Proposition~\ref{lap}), among these functions the odometer is the smallest nonnegative one, so its minimum is $0$.
\end{proof}

Fix $x,z \in V$ and let $f(y) = \frac{g^{z}(x,y)}{\deg(y)}$ be the function satisfying $f(z)=0$ and $\Delta f = \delta_z - \delta_x$. 
(Here $g^z(x,y)$ is the expected number of visits to $y$ by a random walk started at $x$ before hitting $z$).
With a slight abuse of notation, we define $g(x,y): = \sum_{z \in V} \frac{1}{n} g^{z}(x,y)$.

\begin{proof}[Proof of Proposition~\ref{t.finite}]
Observe that $\sum_{x \in V} s(x) = n$.  Therefore $s$ stabilizes to the all $1$ configuration by Lemma~\ref{l.massn}, and the odometer $u$ satisfies
	\[ s+ \Delta u = 1 \]
and $\min u = 0$.

Since $\Delta \frac{g^{z}(x,\cdot)}{\deg(\cdot)} = \delta_{z} - \delta_x$, the function
\begin{eqnarray*}
v^{z}(y) &:=& \frac{1}{\deg(y)}\sum_{x \in V} g^{z}(x,y) (s(x)-1)
\end{eqnarray*}
has $\Delta v^{z}(y) = 1-s(y)$ for $y \neq z$ and
	\[ \Delta v^{z}(z) = \sum_{x \neq z} (s(x)-1) = 1-s(z). \]
Thus $u-v^z$ is harmonic on $V$ and hence is a (random) constant.	

Let $v = \frac{1}{n} \sum_{z \in V} v^{z}$. Since $u-v^z$ is constant for all $z$, the difference $u-v$ is also constant.
Recalling that  $g = \frac{1}{n} \sum_{z \in V} g^{z}$, we have $v(y) = \frac{1}{\deg(y)}\sum_{x \in V} g(x,y)(s(x) - 1)$.  To compute the covariance of the Gaussian vector $v$, note that
	\[ \EE [(s(z)-1)(s(w)-1)] = 1_{\{z=w\}} - \frac1n \]
hence
\begin{align*}
\lefteqn{\mathbb{E}[v(x) v(y)] } \\&= \frac{1}{\deg(x)\deg(y)} \sum_{z,w \in V} g(z, x) g(w,y) \mathbb{E}[(s(z) - 1)(s(w) - 1)] \\
& = \frac{1}{\deg(x)\deg(y)} \left(\sum_{z \in V} g(z,x) g(z,y) - \frac{1}{n} \left( \sum_{z \in V} g(z,x)\right) \left( \sum_{w \in V} g(w,y) \right)\right).
\end{align*}

The function $K(y) :=\frac{1}{\deg(y)}\sum_{w \in V} g(w,y)$ has $\Delta K = \sum_{z,w \in V} \frac{1}{n} \left( \delta_z - \delta_w \right) = 0$, so $K$ is a constant. The second term on the right is just $\frac{K^2}{n}$.
Letting $C$ be a $N(0,\frac{K^2}{n})$ random variable independent of $v$, the Gaussian vectors $\eta$ and $(v(x)+C)_{x \in V}$ have the same covariance matrix, so
	\[ \eta \eqindist v + C. \]
Since $u-v$ is constant and $\min u = 0$ we conclude that \[ u = v - \min v \eqindist \eta - \min \eta. \qedhere \]
\end{proof}

\begin{table}
\begin{center}
\begin{tabular}{|l|l|l|l|l|l|}
\hline
 &  $d=1$ & $d=2$ & $d=3$ & $d=4$ & $d\geq 5$ \\[3pt]
\hline
$\mathbb{E}(\eta_\mathbf 0 - \eta_\mathbf x)^2 \lesssim \psi_d(n,r) := $ & $n r^2$ & $r^2 \log \left( \frac{n}{r} \right)$ & $r$ &  $\log (1+r) $ & $ 1 $ \\[3pt]
\hline
$\EE \max \{ \eta_\mathbf{x} \,:\, \mathbf{x} \in \Z_n^{d} \} \asymp$  & $n^{3/2}$ & $n$ & $n^{1/2}$ & $\log n$ & $(\log n)^{1/2}$ \\[3pt]
\hline
\end{tabular}
\end{center}
\caption{\label{table:orders} Statistics of the bi-Laplacian Gaussian field $\eta$ on the discrete torus $\Z_n^d$. In the first line, $r=\norm{\mathbf x}_2$ and the symbol $\lesssim$ means there is a dimension-dependent constant $C_d$ such that $\EE (\eta_\zero - \eta_\xx)^2 \leq C_d \psi_d(n,r)$ for all $\xx \in \Z_n^d$. The second line gives the order of the expected value of the maximum of the field
up to a dimension-dependent constant factor.
}
\end{table}
	
\section{Green function and bi-Laplacian field on $\Z_n^d$}

The rest of the paper is devoted to the proof of Theorem~\ref{p.torus}. Taking Proposition~\ref{t.finite} as a starting point, the expected odometer equals the expected maximum of the bi-Laplacian Gaussian field $\eta$, since
	\[ \EE u(x) = \EE (\eta_x - \min \eta) = - \EE \min \eta = \EE \max \eta \]
where we have used that $\EE \eta_x=0$. From the covariance matrix for $\eta$ we see that
	\[ \EE (\eta_x - \eta_y)^2 =\sum_{z \in V} \left(\frac{g(z,x)}{\deg(x)}-\frac{g(z,y)}{\deg(y)}\right)^2. \]
We will use asymptotics for the Green function $g$ of the discrete torus $\Z_n^d$ to estimate the right side.
This will enable us to use Talagrand's majorizing measure theorem to determine the order of $\EE \max \{\eta_\mathbf{x} \,:\, \mathbf{x} \in \Z_n^d\} $ up to a dimension-dependent constant factor.
These calculations are carried out below and summarized in Table~\ref{table:orders}.  The table entries give bounds up to a constant factor depending only on the dimension $d$. For example, the $d=3$ column means that there is a positive constant $C$ such that $\EE (\eta_0 - \eta_\mathbf{x})^2 \leq C\norm{\mathbf{x}}_2$ for all $\xx \in \Z_n^3$ and $C^{-1} n^{1/2} \leq \EE \max \eta \leq C n^{1/2}$.

\begin{rem}\label{r-notation}
For the rest of our work, we identify the discrete torus $\mathbb{Z}_n^d$ with $\left(  \mathbb{Z} \cap (-n/2, n/2] \right)^d$ which in turn is viewed as a subset of $\R^d$. For $\mathbf{x} \in \mathbb{Z}_n^d$ and $1 \le p \le \infty$, we denote by $\norm{\mathbf{x}}_p$ the $p$-norm
under the above identification. Note that for standard graph distance $d_G$ on $\mathbb{Z}^d_n$, we have $d_G(\mathbf{0}, \mathbf{x}) = \norm{\mathbf{x}}_1$.
\end{rem}

\subsection{Fourier analysis on the discrete torus}
In this section, we derive a formula for  $\mathbb{E}(\eta_{\mathbf{0}} - \eta_{\mathbf{x}})^2$. We begin by recalling some basic facts about Fourier analysis on the discrete torus and the spectral theory of the Laplacian.
We equip the torus $\mathbb{Z}_n^d$ with normalized Haar measure $\mu$ (in other words the uniform probability measure).
Consider the Hilbert space $\mathcal{H}=L^2( \mathbb{Z}_n^d, \mu)$ of complex valued functions on torus with inner product
\[
\langle f ,g  \rangle = \int_{\mathbb{Z}_n^d} f\bar{ g} \, d\mu = \frac{1}{n^d} \sum_{x \in \mathbb{Z}_n^d} f(x) \overline{g(x)}.
\]
We identify the Pontryagin dual group $\widehat{\mathbb{Z}_n^d}$ with $\mathbb{Z}_n^d$ as follows. For any $\mathbf{a} \in \mathbb{Z}_n^d$, the map $ \mathbf{x} \mapsto \exp\left( i 2 \pi \mathbf{x}\cdot\mathbf{a}/n \right)$
gives the corresponding element in $\widehat{\mathbb{Z}_n^d}$ (The dot product is the usual Euclidean dot product in $\mathbb{R}^n$). We denote this character by $\chi_\mathbf{a}$.
Recall that $\{ \chi_\mathbf{a} : \mathbf{a} \in \mathbb{Z}_n^d \}$ forms an orthonormal basis for $\mathcal{H}$. Moreover each $\chi_\mathbf{a}$ is an eigenfunction for the Laplacian $\Delta$
with eigenvalue
\[
\lambda_\mathbf{a}= -4 \sum_{i=1}^d \sin^2 \left(\frac{\pi a_i}{n} \right).
\]
Thus the Laplacian $\Delta:\mathcal{H} \to \mathcal{H}$ is a non-positive, bounded operator.
Moreover $\lambda_\mathbf{a} = 0$ if and only if $\mathbf{a}= \mathbf{0}$.
Laplacian $\Delta$ is a self-adjoint operator, that is
 \begin{equation} \label{sa}
 \langle f_1 , \Delta f_2 \rangle= \langle \Delta f_1, f_2 \rangle
\end{equation}
for all $f_1,f_2 \in \mathcal{H}$ (See Remark \ref{int-p}).
We denote by $g_\mathbf{x}(\mathbf{y})= g(\mathbf{y},\mathbf{x})$. Recall that
\begin{equation} \label{green}
\Delta g_\mathbf{x} = 2d \left( \frac{1}{n^d} \chi_\mathbf{0}- \delta_\mathbf{x}\right).
\end{equation}
Denote by $\widehat{g_\mathbf{x}}(\mathbf{a})$, the Fourier coefficient $\langle g_\mathbf{x}, \chi_\mathbf{a} \rangle$.
Since the function $\mathbf{x} \mapsto \sum_{\mathbf{y} } g_\mathbf{x}(\mathbf{y})$ is harmonic, it is constant. This implies that there exists $L \ge 0$ such that
\begin{equation}\label{ghat0}
\widehat{g_\mathbf{x}}(\mathbf{0})= n^{-d} \sum_{\mathbf{y} \in \mathbb{Z}^d_n} g_\mathbf x(\mathbf y)= L
\end{equation}
for all $\mathbf{x} \in \mathbb{Z}^d_n$ .
For $\mathbf{a} \neq \mathbf{0}$, we have
\begin{equation}\label{ghata}
 \lambda_\mathbf{a} \widehat{g_\mathbf{x}}(\mathbf{a})= \lambda_\mathbf{a} \langle g_\mathbf{x}, \chi_\mathbf{a} \rangle = \langle g_\mathbf{x}, \Delta \chi_\mathbf{a} \rangle= \langle \Delta g_\mathbf{x},  \chi_\mathbf{a} \rangle = -2d \langle \delta_\mathbf{x},  \chi_\mathbf{a} \rangle= - 2d n^{-d} \chi_{-\mathbf{a}}(\mathbf x).
\end{equation}
For the above equation, we used  $\Delta \chi_\mathbf{a}= \lambda_\mathbf{a} \chi_\mathbf{a}$, equations \eqref{sa}, \eqref{green} and $\langle \chi_\zero, \chi_\mathbf{a}\rangle=0$.
By Parseval's theorem and equations \eqref{ghat0}, \eqref{ghata},
\begin{eqnarray}
\nonumber \mathbb{E} (\eta_\mathbf{0}- \eta_\mathbf{x})^2 &=& (2d)^{-2} \sum_{\mathbf{z} \in \mathbb{Z}^d_n} \left(g(\mathbf{z}, \mathbf{0} )- g(\mathbf{z},\mathbf{x})\right)^2 \\
\nonumber &=& (2d)^{-2} n^d \langle g_\mathbf{0} - g_\mathbf{x},g_\mathbf{0} - g_\mathbf{x} \rangle\\
\nonumber & = & (2d)^{-2}n^d \sum_{\mathbf{z}  \in \mathbb{Z}^d_n} \abs{\widehat{g_\mathbf{0}}(\mathbf{z}) - \widehat{g_\mathbf{x}}(\mathbf{z}) }^2\\
 \label{e.four-form} &=& \frac{1}{4} F_{n,d} (\mathbf x)
\end{eqnarray}
where
\begin{equation} \label{Fnd}
 F_{n,d} (\mathbf x) := n^{-d} \sum_{\mathbf z \in \mathbb{Z}^d_n \setminus \{ \mathbf 0 \} } \frac{ \sin^2 \left( \frac{ \pi \mathbf{x}. \mathbf{z}}{n} \right) } { \left( \sum_{i=1}^d \sin^2 \left( \frac{\pi z_i}{n}\right) \right)^2}
 \end{equation}

 \begin{rem}\label{int-p}
  In $\mathbb{R}^d$, we have the Green's second identity \[\int f_1 \Delta f_2  \, dx = -\int \nabla f_1.\nabla f_2 \, dx = \int f_1 \Delta  f_2 \, dx \] for all $f_1,f_2 \in C_c^\infty (\R^d)$. Similarly, in our discrete setting we have
  \[
    \langle f_1 , \Delta f_2 \rangle= - \frac{1}{2 n^d} \sum_{\mathbf x, \mathbf y \in \mathbb{Z}^d_n} (f_1(\mathbf x)-f_2(\mathbf y))(\overline{f_2(\mathbf x)}-\overline{f_2(\mathbf y)})k(\mathbf x,\mathbf y) = \langle \Delta f_1, f_2 \rangle
  \]
where $k(\mathbf x, \mathbf y)= \mathbf 1_{x \sim y}$ and $\mathbf x \sim \mathbf y$ if $\mathbf x$ and $\mathbf y$ are neighbors in $\mathbb{Z}^d_n$. See for instance, \cite[Lemma 2.1.2]{S-C} or \cite[Lemma 13.11]{LPW} for a proof.
 \end{rem}
Our task now is to estimate the expression $F_{n,d}(\mathbf x)$. Henceforth we assume that $\mathbf{x} \neq \mathbf{0}$.
To study the quantity $\max _{\mathbf{x},\mathbf{y}\in \mathbb{Z}^d_n} (\eta_\mathbf{x} - \eta_\mathbf{y} )$ as $n$ goes to  $\infty$,
we want to estimate $\mathbb{E}(\eta_{\mathbf{0}} - \eta_{\mathbf{x}})^2$ with $d$ fixed and $n$ large for different values of $\mathbf{x}$.
We  approximate $F_{n,d}(\mathbf x)$ by an integral of a function over $\mathbb{R}^d$.
For $\mathbf w \in \mathbb{R}^d$ and $r >0$, we denote by $B_\infty(\mathbf w,r)$ the open ball with center $\mathbf w$ and radius $r$ under supremum norm, that is
\[
 B_\infty(\mathbf w,r) = \{ \mathbf y \in \mathbb{R}^d: \norm{\mathbf y - \mathbf w}_\infty <r \}.
\]
We denote the indicator function of the ball $B_\infty(\mathbf z/n , 1/(2n))$ by $I_{\mathbf z, n}:\mathbb{R}^d \to \{0,1\}$, that is $I_{\mathbf z, n} = \mathbf{1}_{B_\infty(\mathbf z/n , 1/(2n))}$.
Define the function $G_{n,d,\mathbf x} : \mathbb{R}^d \to \mathbb{R}$
\[
G_{n,d,\mathbf x}= \sum_{\mathbf z \in \mathbb{Z}^d_n \setminus \{ \mathbf 0 \} }  \frac{ \sin^2 \left( \frac{ \pi \mathbf{x}\cdot \mathbf{z}}{n} \right) } { \left( \sum_{i=1}^d \sin^2 \left( \frac{\pi z_i}{n}\right) \right)^2} I_{\mathbf z, n} .
\]
Since the cubes $B_\infty(\mathbf z/n , 1/(2n))$ are disjoint with volume $n^{-d}$ , we have
\begin{equation} \label{e.feqintg}
 F_{n,d}(\mathbf x)=  \int_{\mathbb{R}^d} G_{n,d,\mathbf x} (\mathbf y) \,d \mathbf y .
\end{equation}
By triangle inequality, we have
\begin{equation} \label{e.triangle}
(1+ \sqrt{d})^{-1} \norm{\mathbf{z}/n}_2 \le \norm{\mathbf y}_2 \le (1 + \sqrt{d}) \norm{\mathbf z/n}_2
\end{equation}
for all $\mathbf{z} \in \Z^d_n \setminus \{ \zero \}$ and for all $\mathbf y \in B_\infty(\mathbf z/n , 1/(2n))$ under the usual identification from Remark~\ref{r-notation}.
We will estimate the function $G_{n,d,\mathbf x}$ using the function $H_{n,d,\mathbf x} : \mathbb{R}^d \to \mathbb{R}$ defined by
\begin{equation*} \label{defH}
H_{n,d,\mathbf x}(\mathbf y) = \sum_{\mathbf z \in \mathbb{Z}^d_n \setminus \{ \mathbf 0 \} }  \frac{ \sin^2 \left( \frac{ \pi \mathbf{x} \cdot \mathbf{z}}{n} \right) } { \norm{\mathbf y}_2^4 } I_{\mathbf z, n}(\mathbf y) .
\end{equation*}
More precisely, we have the following lemma.
\begin{lem} \label{compLemma}
Fix $ d \in \mathbb{N}^*$. There exist positive reals $c_1,C_1$ such that
\begin{equation} \label{compGH}
 c_1 H_{n,d,\mathbf x}(\mathbf y) \le G_{n,d,\mathbf x}( \mathbf y)\le C_1 H_{n,d,\mathbf x}  (\mathbf y)
\end{equation}
for all $n \in \mathbb{N}^*$, for all $\mathbf x \in \mathbb{Z}^d_n \setminus \{ \mathbf 0 \}$ and for all $\mathbf y \in \mathbb{R}^d$.
\end{lem}
\begin{rem} We will use  $C_i$ for large constants and $c_i$ for small constants. Here and in what follows, all constants are allowed to depend on $d$ but not on $\mathbf{x} \in \mathbb{Z}^d_n \setminus \{ \mathbf 0 \}$ or
 $n$.
\end{rem}
\begin{proof}
 The idea is to use the estimate
 \[
\frac{2}{\pi} \abs{t} \le  \abs{\sin t} \le \abs{t}
 \]
for all $t \in [-\pi/2,\pi/2]$. Thus there exists a constant $C_2>0$ such that
\begin{equation} \label{e-cl1}
C_2^{-1} \norm{\mathbf z/n}_2^4 \le \left( \sum_{i=1}^d \sin^2 \left( \frac{\pi z_i}{n}\right) \right)^2 \le C_2 \norm{\mathbf z/n}_2^4
\end{equation}
for all $\mathbf{z} \in \mathbb{Z}^d_n$ and for all $n \in \mathbb{N}^*$. By \eqref{e.triangle}, we have
\begin{equation} \label{e-cl2}
 (1+\sqrt{d})^{-4} \frac{I_{\mathbf z, n}(\mathbf y)}{ \norm{ \mathbf z/ n }_2^4} \le  \frac{I_{\mathbf z, n}(\mathbf y)}{ \norm{\mathbf y }_2^4}  \le (1+\sqrt{d})^4 \frac{I_{\mathbf z, n}(\mathbf y)}{ \norm{ \mathbf z/ n }_2^4}
\end{equation}
for all $\mathbf{z} \in \mathbb{Z}^d_n \setminus \{ \mathbf 0 \}$, for all $\mathbf y \in \mathbb{R}^d$ and for all $n \in \mathbb{N}^*$. Combining equations \eqref{e-cl1} and \eqref{e-cl2} gives \eqref{compGH}.
\end{proof}
By \eqref{e.four-form}, \eqref{e.feqintg} along with integration of  \eqref{compGH} over the variable $\mathbf y$, there exists $c_1,C_1>0$ such that
\begin{equation} \label{compFeta}
 c_1 d^2 \int_{\mathbb R^d }H_{n,d,\mathbf x}(\mathbf y) \, d\mathbf y \le \mathbb{E} (\eta_\mathbf 0 -\eta_\mathbf x)^2 \le C_1  d^2 \int_{\mathbb R^d} H_{n,d,\mathbf x}  (\mathbf y) \, d \mathbf y
\end{equation}
for all $n \in \mathbb{N}^*$ and for all $\mathbf x \in \mathbb{Z}^d_n \setminus \{ \mathbf 0 \}$.

By \eqref{compFeta}, it suffices to estimate $\int_{\mathbb{R}^d}H_{n,d,\mathbf{x}} (\mathbf y) \,d\mathbf y$. Observe that the support of $H_{n,d,\mathbf{x}}$ satisfies
\[
\operatorname{Support}( H_{n,d,\mathbf{x}}) \subseteq B_2(\mathbf 0 , \sqrt d) \setminus B_2( \mathbf 0, 1/(2n))
\]
for all $d,n \in \mathbb{N}^*$ and for all $\mathbf x \in \mathbb{Z}^d_n$,
where $B_2$ denotes open ball with respect to Euclidean norm in $\mathbb{R}^d$.

\subsection{Upper bounds}
Define $\psi_d$  by
 \begin{equation} \label{e.psi}
 \psi_d(n,r):= \begin{cases}
                          n r^2   & \mbox{if } d=1 \\
                          r^2 \log \left( \frac{n}{r } \right) & \mbox{if } d=2 \\
                           r & \mbox{if } d=3 \\
                              \log (1+r)  & \mbox{if } d=4  \\
                             1 & \mbox{if } d \ge 5.                         \end{cases}
\end{equation}
for all $n \in \N^*$ and all $r>0$ along with $\psi_d(n ,0):=0$ for all $d,n \in \N^*$.
 The upper bounds for $ \mathbb{E}(\eta_{\mathbf{0}} - \eta_{\mathbf{x}})^2$ is summarized in the following Proposition.
\begin{proposition}\label{p.r1ub}
 For each $d \in \N^*$, there exists $C_d >0$ such that
 \begin{equation} \label{e.r1ub}
  \mathbb{E}(\eta_{\mathbf{0}} - \eta_{\mathbf{x}})^2 \le C_d \psi_d(n, \norm{\mathbf x}_2)
 \end{equation}
for all $n \in \N^*$ and for all $\mathbf x \in \mathbb{Z}_n^d$ , where $\psi_d$ is defined by \eqref{e.psi}.
\end{proposition}
\begin{proof}
 By \eqref{compFeta} it suffices to find upper bounds for $\int_{\mathbb{R}^d} H_{n,d,\mathbf x}(\mathbf y) \, d \mathbf y$.
The strategy to establish upper bounds for $\int_{\mathbb{R}^d} H_{n,d,\mathbf x}(\mathbf y) \, d \mathbf y$ is to split it into two integrals as
$ \int_{\mathbb{R}^d} =   \int_{ 1/(2n) \le \norm{\mathbf y}_2 \le \sqrt{d}/ \norm{\mathbf x}_2}  +\int_{ \sqrt{d}/\norm{\mathbf x}_2 < \norm{\mathbf y}_2 \le \sqrt{d} }$.
Note that both the integrals are over non-empty annuli since $1/(2n) \le \sqrt{d}/(4\norm{\mathbf x}_2) \le  \sqrt{d}/\norm{\mathbf x}_2 \le \sqrt{d}$ for all $\mathbf{x} \in \Z^d_n \setminus \{ \zero \}$ with the identification from
Remark~\ref{r-notation}.
Using Cauchy-Schwarz inequality and the bound $\abs{\sin t } \le t$, we have $\abs{\sin (\pi \mathbf{x}. \mathbf{z/n})} \le \pi \norm{\mathbf x}_2 \norm{\mathbf z/n}_2 $.
This bound competes with the trivial bound $\abs{\sin (\pi \mathbf{x}. \mathbf{z/n}) } \le 1$. It will become clear that up to constants, the first bound is better for the first term
and the trivial bound  $\abs{\sin (\pi \mathbf{x}. \mathbf{z/n}) } \le 1$ is better for the second term.

For the first integral we use the bound $\abs{\sin t} \le \abs{t}$ and Cauchy-Schwarz inequality  to obtain
\[
H_{n,d,\mathbf x}(\mathbf y)  \le \sum_{\mathbf z \in \mathbb{Z}^d_n \setminus \{ \mathbf 0 \} }  \frac{\pi^2 \norm{\mathbf{x}}_2^2 \norm{ \mathbf{z}/n}_2^2  } { \norm{\mathbf y}_2^4 } I_{\mathbf z, n}(\mathbf y).
\]
By \eqref{e.triangle}, we have $\norm{\mathbf z/n}_2^2 I_{\mathbf z , n} (\mathbf y) \le (1 +\sqrt{d})^2 \norm{\mathbf y}_2^2 I_{\mathbf z, n}(\mathbf y)$. Therefore, we obtain
\begin{equation} \label{1d-ub}
 H_{n,d,\mathbf x}(\mathbf y)  \le \sum_{\mathbf z \in \mathbb{Z}^d_n \setminus \{ \mathbf 0 \} }  \frac{(1 +\sqrt{d})^2  \pi^2 \norm{\mathbf{x}}_2^2 } { \norm{\mathbf y}_2^2 } I_{\mathbf z, n}(\mathbf y).
\end{equation}
for all $n,d \in \mathbb{N}^*$, for all $\mathbf y \in \mathbb{R}^d$ and for all $\mathbf x \in \mathbb{Z}^d_n \setminus \{ \mathbf 0 \}$.
Hence, we have
\begin{equation} \label{1t-ub}
 I_1 := \int_{ 1/(2n) \le \norm{\mathbf y}_2 \le \sqrt{d}/ \norm{\mathbf x}_2} H_{n,d,\mathbf x}(\mathbf y) \, d \mathbf{y} \le (1 +\sqrt{d})^2  \pi^2 \norm{\mathbf x}_2^2 \omega_{d-1} \int_{1/(2n)}^{\sqrt{d}/ \norm{\mathbf x}_2} \frac{r^{d-1}}{r^2} \, dr
\end{equation}
where $\omega_{d-1}=\frac{2 \pi^{d/2} }{\Gamma (d/2)} $ is the  $(d-1)$-dimensional surface measure of unit sphere $\mathbb{S}^{d-1}$ in $\mathbb{R}^d$.

For the second integral we use the bound $\abs{\sin t} \le 1$, to obtain
\begin{equation} \label{2t-ub}
 I_2:=\int_{ \sqrt{d}/\norm{\mathbf x}_2 < \norm{\mathbf y}_2 \le \sqrt{d} } H_{n,d,\mathbf x}(\mathbf y) \, d \mathbf{y} \le \omega_{d-1} \int_{\sqrt{d}/ \norm{\mathbf x}_2}^ {\sqrt{d}} \frac{r^{d-1}}{r^4} \, dr.
\end{equation}
Combining equations \eqref{1t-ub} and \eqref{2t-ub}, we obtain
\begin{equation} \label{234ub}
 \int_{\mathbb{R}^d} H_{n,d,\mathbf x}(\mathbf y) \, d \mathbf y  \le  (1 +\sqrt{d})^2  \pi^2 \norm{\mathbf x}_2^2 \omega_{d-1} \int_{1/(2n)}^{\sqrt{d}/ \norm{\mathbf x}_2} r^{d-3} \, dr+  \omega_{d-1}\int_{\sqrt{d}/ \norm{\mathbf x}_2}^ {\sqrt{d}} r^{d-5} \, dr.
\end{equation}
The desired upper bounds on $\mathbb{E}(\eta_{\mathbf{0}} - \eta_{\mathbf{x}})^2$ for all dimensions follow from  \eqref{234ub} along with \eqref{compFeta}.
\end{proof}

\old{ 
$\mathbf{ d=1}$.
There exists $C_1>0$ such that
\begin{equation*} \label{ub1}
\mathbb{E}(\eta_{\mathbf{0}} - \eta_{\mathbf{x}})^2 \le C_1 n \norm{\mathbf x}_2^2   .
\end{equation*}
for all $n \in \mathbb{N}^*$ and for all $x \in \mathbb{Z}_n$. \\
$\mathbf{ d=2}$.\\
There exists $C_2>0$ such that
\begin{equation*} \label{ub2}
  \mathbb{E}(\eta_{\mathbf{0}} - \eta_{\mathbf{x}})^2 \le C_2  \norm{\mathbf x}_2^2 \log \left( \frac{n}{\norm{\mathbf x}_2 }\right)
\end{equation*}
for all  $n \in \mathbb{N}^*$ and for all $\mathbf x \in \mathbb{Z}_n^2$. \\
$\mathbf{ d=3}$.\\
There exists $C_3>0$ such that
\begin{equation*} \label{ub3}
 \mathbb{E}(\eta_{\mathbf{0}} - \eta_{\mathbf{x}})^2 \le C_3 \norm{\mathbf x}_2
\end{equation*}
for all  $n \in \mathbb{N}^*$ and for all $\mathbf x \in \mathbb{Z}_n^3$. \\
$\mathbf{ d=4}$.\\
There exists $C_4>0$ such that
\begin{equation*} \label{ub4}
  \mathbb{E}(\eta_{\mathbf{0}} - \eta_{\mathbf{x}})^2 \le C_4 \left( 1+ \log \norm{\mathbf x}_2 \right)
\end{equation*}
for all  $n \in \mathbb{N}^*$ and for all $\mathbf x \in \mathbb{Z}_n^4$. \\
$\mathbf{ d \ge 5}$.\\
There exists $C_5>0$ such that
\begin{equation*} \label{ub5}
 \mathbb{E}(\eta_{\mathbf{0}} - \eta_{\mathbf{x}})^2    \le  C_5 \omega_{d-1} d^{d/2}.
\end{equation*}
for all $d \in \mathbb{N}^*$ with $d \ge 5$ , for all $n \in \mathbb{N}^*$ and for all $\mathbf x \in \mathbb{Z}_n^d$
}
\begin{rem}
 The terms $I_1$ and $I_2$ correspond to the energy (square of 2-norm) of the low and high frequency oscillations of the function $g_\mathbf 0 - g_\mathbf x$ respectively.
 For $d=1,2$, the term $I_1$ dominates $I_2$.
 For $d=3$, both $I_1$ and $I_2$ are of the same order. For $d \ge 4$, the term $I_2$ dominates $I_1$.
Hence our approach to obtain matching lower bounds in the next subsection is as follows:  For $d=1,2$, we obtain lower bounds on lower frequency terms and for $d \ge 3$ we obtain lower bounds on higher frequency terms.
\end{rem}
It is well know that the  Gaussian field $\eta$  induces a Hilbert space on $\Z^d_n$ given by the distance metric
\[
 d_\eta(\mathbf{x},\mathbf{y}) := \left( \EE (\eta_\mathbf{x} - \eta_\mathbf{y})^2 \right)^{1/2}.
\]
The upper bounds on $d_\eta$ provided by Proposition \ref{p.r1ub} transfers to upper bounds on $\EE \sup_{\mathbf x \in \Z^d_n} \eta_{\mathbf x}$. The main tool to transfer bounds is
Dudley's bound \cite[Proposition 1.2.1]{Tal} described below. There exists $L>0$ such that
\begin{equation} \label{e.dud}
 \EE \sup_{\mathbf x \in \Z^d_n} \le L \sum_{k=0}^\infty 2^{k/2} e_k
\end{equation}
where $e_k = \inf \sup_{\mathbf t \in \Z^d_n} d_\eta(\mathbf t,T_k)$ and the infimum is taken over all subsets $T_k \subseteq \Z^d_n$ with $\abs{T_k} \le \ttk$.

\begin{proposition}\label{p.r2ub}
 For each $d \in \N^*$, there exists $C_d >0$ such that
 \begin{equation} \label{e.r2ub}
\EE \sup_{\mathbf x \in \Z^d_n} \eta_{\mathbf x} \le C_d \phi_d(n)
 \end{equation}
for all $n \in \N^*$, where $\phi_d$ is defined by \eqref{e.phi}.
\end{proposition}

\begin{proof}
Let $d_G$ denote the standard graph distance on the torus $\Z^d_n$.
By choosing a submesh of appropriate cardinality the following statement is clear:
For any $d \in \N^*$, there exist $C_{d,1} >0$ such that for any $n \ge 2$ and for any $ 2 \le m < n^d$, there exists a set $S_m \subset \Z^d_n$ with $\abs{S_m}=m$ such that
\begin{equation} \label{e.sep}
  \sup_{\mathbf t \in \Z_n^d} d_G(t,S_m)= \sup_{\mathbf t \in \Z_n^d} \inf_{\mathbf s \in S_m} d_G(\mathbf t,\mathbf s) \le C_{d,1} \frac{n}{m^{1/d}}.
\end{equation}
For each $d \in \N^*$ by Dudley's bound \eqref{e.dud}, \eqref{e.sep} and Proposition \ref{p.r1ub}, there exists $C_{d,2}, C_{d,3} >0$
\begin{equation} \label{e.ubm}
 \EE \max_{\mathbf{x} \in \Z^d_n} \eta_x \le  C_{d,2}  \sum_{k=0}^{\lfloor \log \log n^d \rfloor}  2^{k/2} \left[ \psi_d\left(n, \frac{C_{d,1} n}{2^{(2^k/d)}}\right) \right]^{1/2} \le C_{d,3} \phi_d(n)
\end{equation}
The second inequality above follows from a straightforward case by case calculation.
\end{proof}

\subsection{Lower bounds}
Next, we  prove matching lower bounds on $\EE (\eta_{\mathbf{0}} - \eta_{\mathbf{x}} )^2$.
For dimensions $d=1,2$, we  estimate $F_{n,d}$ directly. \\
\paragraph{$\mathbf{ d=1}$:}
We use the bound $\abs{t} \ge \abs{ \sin t } \ge \frac{2}{\pi} \abs{t}$ for all $\abs{t} \le \pi/2$ to obtain
\[
 F_{n,1} (\mathbf x) \ge n^{-1} \frac{\sin^2 \frac{\pi \mathbf x}{n}}{\sin^4 \frac{\pi}{n}} \ge 4 \pi^{-4} n\norm{ \mathbf x}_2^2
\]
for all $n \in \mathbb{N}^*$ and for all $\mathbf x \in \mathbb{Z}_n \setminus \{ \mathbf 0 \}$. Hence there exists
 $c_1>0$ such that
\begin{equation} \label{lb1}
 \mathbb{E}(\eta_{\mathbf{0}} - \eta_{\mathbf{x}})^2 \ge c_1 n \norm{\mathbf x}_2^2
\end{equation}
for all $n \in \mathbb{N}^*$ and for all $\mathbf x \in \mathbb{Z}_n \setminus \{ \mathbf 0 \}$. \\

\paragraph{$\mathbf{ d=2}$:}
Let $S_{k} \subset \mathbb{Z}^2$ denote the sphere with center $\mathbf 0$ and radius $k$ in the supremum norm, that is $S_{k} = \{ \mathbf y \in \mathbb{Z}^2:\norm{\mathbf y}_\infty =k \}$.
For $x \in \mathbb{R}^2$, we define $H_{\mathbf x}= \{ \mathbf y \in \mathbb{Z}^2: \abs{\mathbf x \cdot \mathbf y} \ge \norm{\mathbf x}_2 \norm{\mathbf y}_2/\sqrt{2} \}$.
It is easy to check that $\abs{S_k} =4k$ and $\abs{S_k \cap H_\mathbf x }\ge 2k$ for all $k \in \mathbb{N}^*$ and for all $\mathbf x \in \mathbb{R}^2$. Let $\alpha \in (1, \sqrt{2})$.
If $\norm{\mathbf z}_2 \le \frac{n}{\alpha \norm{\mathbf x}_2} $ by Cauchy-Schwarz inequality we have $\abs{\mathbf x \cdot \mathbf z/n} \le \alpha^{-1}$.
We need the inequality $\pi \abs{t} \ge \abs{\sin \pi t} \ge \beta \abs{t}$ for all $\abs{t} \le \alpha ^{-1}$ where $\beta= \alpha \sin (\pi \alpha^{-1})$. Putting together the above pieces, we obtain
\begin{eqnarray*}
 F_{n,2}(\mathbf x) &\ge&  \sum_{k=1}^{\left\lfloor \frac{n}{\alpha \norm{\mathbf x}_2} \right\rfloor \wedge \lfloor n/4 \rfloor}   \sum_{\mathbf z \in S_k } \beta^2 \pi^{-4} \norm{\mathbf z}_2^{-4} \abs{\mathbf x \cdot \mathbf z}^2 \\
 & \ge &  \frac{\beta^2}{2 \pi^4} \norm{\mathbf x}_2^2 \sum_{k=1}^{\left\lfloor \frac{n}{\alpha \norm{\mathbf x}_2} \right\rfloor \wedge \lfloor n/4 \rfloor}   \sum_{\mathbf z \in S_k \cap H_\mathbf x}\norm{\mathbf z}_2^{-2}  \\
 & \ge &  \frac{\beta^2}{4 \pi^4} \norm{\mathbf x}_2^2 \sum_{k=1}^{\left\lfloor \frac{n}{\alpha \norm{\mathbf x}_2} \right\rfloor\wedge \lfloor n/4 \rfloor}   \sum_{\mathbf z \in S_k \cap H_\mathbf x} k^{-2}  \\
 & \ge &  \frac{\beta^2}{2 \pi^4} \norm{\mathbf x}_2^2 \sum_{k=1}^{\left\lfloor \frac{n}{\alpha \norm{\mathbf x}_2} \right\rfloor\wedge \lfloor n/4 \rfloor}    k^{-1}  \\
 & \ge &  \frac{\beta^2}{2 \pi^4} \norm{\mathbf x}_2^2\log \left( \left( \left\lfloor \frac{n}{\alpha \norm{\mathbf x}_2} \right\rfloor \wedge \lfloor n/4 \rfloor \right)+1 \right)  \\
 & \ge &  \frac{\beta^2}{2 \pi^4} \norm{\mathbf x}_2^2\log \left(  \frac{n}{\alpha \norm{\mathbf x}_2} \wedge  \frac{n}{4} \right).
\end{eqnarray*}
Since $\norm{\mathbf x}_2 \le n/\sqrt{2}$ for all $\mathbf x \in \mathbb{Z}^2_n$, we have the desired lower bound by using $\alpha < \sqrt{2}$. That is, there exists
 $c_2>0$ such that
\begin{equation} \label{lb2}
 \mathbb{E}(\eta_{\mathbf{0}} - \eta_{\mathbf{x}})^2 \ge c_2  \norm{\mathbf x}_2^2\log \left( \ \frac{n}{ \norm{\mathbf x}_2}  \right)
\end{equation}
for all $n \in \mathbb{N}^*$ with $n \ge 4$ and for all $\mathbf x \in \mathbb{Z}^2_n \setminus \{ \mathbf 0 \}$.
\\

\paragraph{$\mathbf{d \geq 3}$:} For dimensions $d \ge 3$, we will approximate $H_{n,d,\mathbf x}$ by its almost everywhere point-wise limit $H_{\infty,d,\mathbf x}$ defined by
\[
 H_{\infty,d,\mathbf x} (\mathbf y) = \frac{ \sin^2( \pi \mathbf x \cdot \mathbf y) }{ \norm{\mathbf y}_2^4} \mathbf 1_{B_\infty( \mathbf 0,1/2) }(\mathbf y).
\]
for $ y \neq  \mathbf 0$ and $0$ otherwise.
Therefore we would like to estimate integrals of the form
\begin{equation} \label{e-hd1}
 \int_{r_1 \le \norm{\mathbf y}_2 \le r_2}  \frac{ \sin^2( \pi \mathbf x \cdot \mathbf y) }{ \norm{\mathbf y}_2^4}  \, d\mathbf y= \int_{r_1}^{r_2} r^{d-5} s_d(r \norm{\mathbf x}_2 )  \, d r
\end{equation}
where $s_d(t) = \int_{\mathbb{S}^{d-1}} \sin^2(\pi t y_1 ) \, \nu_{d-1}(d\mathbf y)$ and $\nu_{d-1}$ denotes the surface measure in $\mathbb{S}^{d-1}$. We will need the following lower bound for $s_d$.
\begin{lem}\label{lem-sine} Fix $ d\in \mathbb{N}^*$ with $d \ge 3$. Then
 for all $\epsilon >0$, there exists $\delta>0$ such that $s_d(t) \ge \delta$ for all $t \ge \epsilon$.
\end{lem}
\begin{proof}
 Since $s_d:\mathbb{R} \to \mathbb{R}$ is a continuous function with $s_d(t) > 0$ for all $t \neq 0$, it suffices to show that $\liminf_{t \to \infty} s_d(t) >0$.
 By \cite[Corollary 4]{BGMN} (See Remark \ref{cor-4}(b)),
  \[
   s_d(t) = c_d \int_{-1}^1 (1-x^2)^{(d-3)/2}  \sin^2 (\pi t w)   \, dw \ge c_d 2^{(3-d)/2} \int_{-1/2}^{1/2} \sin^2 (\pi t w) \, dw
  \]
where $c_d$ is a constant that depends on $d$.
Since $\lim_{t \to \infty }  \int_{-1/2}^{1/2} \sin^2 (\pi t w) \, dw = 1/2$, the conclusion follows.
\end{proof}
\begin{rem} \label{cor-4}
\begin{itemize}
 \item[(a)]   Recall that we used the point-wise bound $\abs{\sin t} \le 1$ to obtain upper bounds on $I_2$. We want to somehow reverse that inequality  to obtain corresponding lower bounds.  Although the reverse inequality $\abs{\sin t} > \delta$ is not true for any $\delta>0$ in a pointwise sense,
  it is true  in an average sense. That is the content of Lemma \ref{lem-sine}.
\item[(b)] \cite[Corollary 4]{BGMN} implies the following striking result in geometric probability: Let $d \ge 3$. For a uniformly distributed random vector $\mathbf y=(y_1,y_2,\ldots,y_d)$ in the $(d-1)$-dimensional unit sphere $\mathbb{S}^{d-1}$ in $\mathbb{R}^d$, the projection
 $(y_1,y_2, \ldots, y_{d-2})$ is uniformly distributed in the $(d-2)$-dimensional unit ball $\mathbb{B}^{d-2}= B_2(\mathbf 0 , 1)$ in $\mathbb{R}^{d-2}$.
\item[(c)]  Since $\lim_{n \to \infty} H_{n,d,\mathbf x}= H_{\infty , d, \mathbf x}$ almost everywhere, one might wonder if we can prove matching lower bounds for $I_2$ using dominated convergence theorem.  This approach gives a lower
 bound as $n$ goes to $\infty$ but with both $d$ and $\mathbf x$ fixed. However we want lower bounds with fixed $d$ and with both $n$ and $\mathbf x$ varying. Hence there is a need to quantify this convergence as both $n$ and $\mathbf x$ varies.
 We fulfill this need in Lemma \ref{lem-conv}.
\end{itemize}
\end{rem}

One can easily check that $\lim_{n \to \infty} H_{n,d,\mathbf x}= H_{\infty , d, \mathbf x}$ almost everywhere. We need the following quantitative version of this convergence.
\begin{lem}\label{lem-conv}
 Fix $d \in \mathbb{N}^*$.
  For any $\epsilon >0$, there exists positive reals $\delta, N $ such that
\[
 \abs{H_{n,d,\mathbf{x}}(\mathbf y) - \frac{ \sin^2( \pi \mathbf x \cdot \mathbf y) }{ \norm{\mathbf y}_2^4}}  \le \frac{\epsilon}{\norm{ \mathbf y}_2^4}
\]
for all $n \ge N$, for all $\mathbf x \in \mathbb{Z}^d_n \setminus \{ \mathbf 0 \}$  with $\norm{\mathbf x}_2< \delta n$ and for almost every $y \in B_2(\mathbf 0,1/4) \setminus B_2 (\mathbf 0, 1/(8 \norm{\mathbf x}_2))$.
\end{lem}
\begin{proof}
Note that, we have the inclusion
\[
B_\infty (\mathbf 0,1/(2n)) \subset B_2( \mathbf 0, 1/(8 \norm{\mathbf x}_2)) \subset B_2(\mathbf 0 , 1/4)
 \]
 for all $n \ge 4 \sqrt{d}$ and for all $\mathbf x \in \mathbb{Z}^d_n \setminus \{ \mathbf 0 \}$ with $\norm{ \mathbf x}_2 \le  n/ 4 \sqrt{d}$.
We use  $\norm{\mathbf x}_2 \ge 1$  and the comparison of norms $\norm{\mathbf w}_2 \le \sqrt{d} \norm{\mathbf w}_\infty$ to prove the above inclusions.
The function $\mathbf y \mapsto \sin^2(\pi \mathbf x \cdot \mathbf y)$ has gradient bounded uniformly in 2-norm by  $\pi \norm{\mathbf x}_2$. Hence we have
\[
\abs{H_{n,d,\mathbf{x}}(\mathbf y) - \frac{ \sin^2( \pi \mathbf x \cdot \mathbf y) }{ \norm{\mathbf y}_2^4}}  \le \pi \norm{\mathbf x}_2 \frac{\sqrt{d}}{n} \norm{\mathbf y}_2^{-4}
\]
 for all $n \ge 4 \sqrt{d}$ and for all $\mathbf x \in \mathbb{Z}^d_n \setminus \{ \mathbf 0 \}$ with $\norm{ \mathbf x}_2 \le  n/ 4 \sqrt{d}$ and for almost every $y \in B_2(\mathbf 0,1/4) \setminus B_2 (\mathbf 0, 1/(8 \norm{\mathbf x}_2))$.
 The choice  $\delta = \min\left( \frac{1}{4 \sqrt{d}}, \frac{\epsilon}{ \pi \sqrt{d}}\right)$ and $N= 4 \sqrt{d}$ satisfies all the requirements.
 \end{proof}
 We put together the above pieces to obtain the following lower bound for $d \ge 3$.
\begin{lem}\label{lem-highd-lb}
Fix $d \ge 3$. There exists positive reals $\delta,N, c_d$ such that
\[
 \int_{\mathbb{R}^d} H_{n,d,\mathbf x} (\mathbf y) \, d\mathbf y \ge c_d  \int_{1/(8 \norm{\mathbf{x}}_2)}^{1/4} r^{d-5} \, dr
\]
for all $n \ge N$ and for all $\mathbf x \in \mathbb{Z}^d_n \setminus \{ \mathbf 0 \}$  with $\norm{\mathbf x}_2< \delta n$.
\end{lem}
\begin{proof}
By Lemma \ref{lem-sine}, there exists $\epsilon_1>0$ such that $s_d(t) \ge 2 \epsilon_1$ for all $t \ge 1/8$. By Lemma \ref{lem-conv}, there exists positive reals $\delta, N $ such that
\[
 \abs{H_{n,d,\mathbf{x}}(\mathbf y) - \frac{ \sin^2( \pi \mathbf x \cdot \mathbf y) }{ \norm{\mathbf y}_2^4}}  \le \frac{\epsilon_1 }{\omega_{d-1} \norm{ \mathbf y}_2^4}
\]
for all $n \ge N$, for all $\mathbf x \in \mathbb{Z}^d_n \setminus \{ \mathbf 0 \}$  with $\norm{\mathbf x}_2< \delta n$ and for almost every $y \in B_2(\mathbf 0,1/4) \setminus B_2 (\mathbf 0, 1/(8 \norm{\mathbf x}_2))$.

Combining the above observations, we have for all $n \ge N$ and for all $\mathbf x \in \mathbb{Z}^d_n \setminus \{ \mathbf 0 \}$  with $\norm{\mathbf x}_2< \delta n$
\begin{eqnarray*}
\int_{\mathbb{R}^d}  H_{n,d,\mathbf x} (\mathbf y) \, d\mathbf y & \ge & \int_{1/(8\norm{\mathbf{x})}_2 \le \norm{\mathbf y}_2 \le 1/4}  H_{n,d,\mathbf x} (\mathbf y) \, d\mathbf y \\
& \ge & \int_{1/(8\norm{\mathbf{x})}_2 \le \norm{\mathbf y}_2 \le 1/4} ( H_{\infty,d,\mathbf x} (\mathbf y) - \epsilon_1 \omega_{d-1}^{-1}  \norm{\mathbf y}_2^{-4})\, d\mathbf y \\
& = & \int_{1/(8 \norm{\mathbf{x}}_2)}^{1/4} r^{d-5} (s_d(r \norm{\mathbf x}_2) - \epsilon_1) \, dr \\
& \ge & \epsilon_1 \int_{1/(8 \norm{\mathbf{x}}_2)}^{1/4} r^{d-5}  \, dr . \qedhere
\end{eqnarray*}
\end{proof}
We now establish the following lower bounds corresponding to the upper bounds in Proposition \ref{p.r1ub}.
\begin{proposition} \label{p.r1lb}
 For each $d \in \N^*$, there exists positive reals $\delta_d,N_d,c_d$ such that
\[
 \mathbb{E}(\eta_{\mathbf{0}} - \eta_{\mathbf{x}})^2  \ge c_d  \psi_d(n, \norm{\mathbf x}_2)
\]
for all $n \ge N_d$ and for all $\mathbf x \in \mathbb{Z}^d_n \setminus \{ \mathbf 0 \}$  with $\norm{\mathbf x}_2< \delta_d n$, where $\psi_d$ is defined by \eqref{e.psi}.
\end{proposition}
\begin{proof}
The cases $d=1,2$ follow from \eqref{lb1} and \eqref{lb2} respectively. The case $d \geq 3$ follows from
 Lemma \ref{lem-highd-lb} along with \eqref{compFeta}.
\end{proof}

\begin{rem}\label{rem-sketch}
The condition $\norm{\mathbf x}_2 < \delta_d n$ that appears in the lower bound for the case $d \ge 3$ is somewhat unsatisfactory.
 We believe that the lower bound is true without any such an additional condition.
 However the lower bounds in the present form are good enough for our main application.
\old{\item[(b)]
 Recall that the reason behind the estimation of $\mathbb{E} (\eta_\mathbf 0 - \eta_\mathbf x)^2 $ is to estimate the quantity
 $M_{d}(n)=\mathbb E \max_{\mathbf x} \eta_{\mathbf x}$. Upper bounds on $\mathbb{E} (\eta_\mathbf 0 - \eta_\mathbf x)^2 $
 translate to upper bounds on $M_d(n)$ by Talagrand's equation 1.12. Suppose there is a method to give lower bounds  $M_{d}(n) \ge l_d(n)$ (for some function $l_d$) using \emph{strong} lower bounds on $\mathbb{E} (\eta_\mathbf 0 - \eta_\mathbf x)^2 $ without
 the $\norm{x}_2 < \delta n $ condition. Then the same method can be used to give lower bounds on $M_{d}(n)$ using our \emph{weaker} lower bounds. We simply use the inequality
 \[
M_{d}(n) \ge \mathbb E \max_{\mathbf x,\mathbf y \in \mathbb{Z}^d_n \cap B_\infty( \mathbf 0 , \frac{\delta}{2 \sqrt{d} } n)} \eta_{\mathbf x}  \ge l_d (\delta n/2 \sqrt d).
 \]
This will give the same lower bounds assuming $l_d(n)$ and $l_d (\delta n/2 \sqrt d)$ are of the same order as a function of $n$.
\end{itemize} }
\end{rem}

Next, we obtain lower bounds matching the upper bounds in Proposition \ref{p.r2ub}. We start by recalling notation and setup for Talagrand's majorizing measure \cite[Theorem 2.1.1]{Tal}.
We consider centered multivariate Gaussian random variables $(\eta_t)_{t \in T}$ indexed by a set $T$ with cardinality $\abs{T}$. An \emph{admissible sequence}   $\{ \mathcal{A}_k \}$ is an increasing sequence of partitions of $T$ such that $|\mathcal{A}_k| \leq \ttk$.
Here ``increasing sequence" refers to the fact that every set in $\mathcal{A}_{n+1}$ is contained in a set in $\mathcal{A}_n$.
We denote by $A_k(t)$ the unique element of $\mathcal{A}_n$ that contains $t \in T$. Recall that $d_\eta(t_1,t_2) = \left( \EE (\eta_{t_1}-\eta_{t_2})^2 \right)^{1/2}$ denotes
the Hilbert space metric induced by $(\eta_t)_{t \in T}$. We define the function
\[ \gamma_2 (T,d_\eta) = \inf \sup_{t \in T} \sum_{k=0}^\infty \operatorname{diam}_\eta(A_n(t))\]
where $\operatorname{diam}_\eta$ denotes the diameter in the $d_\eta$ metric and the infimum is taken over all admissible sequences.
The majorizing measure theorem  \cite[Theorem 2.1.1]{Tal} states that there is
 some universal constant $L$ for which
\begin{equation} \label{e.majormeas}
\frac{1}{L} \gamma_2(T, d) \leq \mathbb{E} \sup_{t \in T} \eta_t \leq L \gamma_2(T, d).
\end{equation}

\begin{proposition} \label{p.r2lb}
 For each $d \in \N^*$, there exists $c_d >0$ such that
 \begin{equation} \label{e.r2lb}
\EE \sup_{\mathbf x \in \Z^d_n} \eta_{\mathbf x} \ge c_d \phi_d(n)
 \end{equation}
for all $n \in \N^*$, where $\phi_d$ is defined by \eqref{e.phi}.
\end{proposition}
\begin{proof}
 The strategy is to use the lower bound for $d_\eta$ given by Proposition \ref{p.r1lb} along with \eqref{e.majormeas} where $T= \{ \mathbf x \in \Z^d_n: \norm{\mathbf{x}}_2 < \delta_d n \} $ and
 $\EE \sup_{t \in \Z^d_n} \eta_t \ge \EE \sup_{t \in T} \eta_t $. Note that it suffices to  show \eqref{e.r2lb} for large enough $n$, \emph{i.e.} $n > N_d$ for some fixed $N_d$.

 For $d=1,2,3$, by \eqref{e.majormeas}  we have
 \begin{equation} \label{e.llb0}
  \EE \sup_{t \in \Z^d_n} \eta_t \ge L^{-1} \inf \sup_{t \in T} \operatorname{diam}_\eta (A_0(t)).
 \end{equation}
Since $\abs{ \mathcal{A}_0} \le 2$, we have $ \sup_{t \in T}\operatorname{diam}_G (A_0(t)) \ge c_0 n$ for some $c_0>0$. Therefore by Proposition \ref{p.r1lb} along with \eqref{e.llb0}  we obtain the desired result.

For $d=4$, by \eqref{e.majormeas} we have
\begin{equation}\label{e.llb1}
   \EE \sup_{t \in \Z^d_n} \eta_t \ge L^{-1} \inf \sup_{t \in T} 2^{k/2} \operatorname{diam}_\eta (A_k(t)).
\end{equation}
where $k = \left\lfloor\log_2 \log_2 \abs{T}\right\rfloor-1$. This gives $2^{k/2}  \ge c_0 \sqrt{\log n}$ for some $c_0 >0$. Moreover, $k = \left\lfloor\log_2 \log_2 \abs{T}\right\rfloor-1$ and $\abs{\mathcal{A}_k} \le \ttk$
implies that at least one of the sets $A_k(t)$ has cardinality greater than or equal to $\sqrt{\abs{T}}$, which in turn implies $\operatorname{diam}_G(A_k(t)) \ge c_1 \sqrt{n}$ for some $c_1>0$.
By Proposition \ref{p.r1lb}, we obtain $\operatorname{diam}_G(A_k(t)) \ge c_2 \sqrt{\log n}$ for some $c_2>0$ and for large enough $n$. The conclusion for $d=4$ then follows from \eqref{e.llb1}.

 The case $d \ge 5$ is a direct consequence of Sudakov minoration (\cite[Lemma 2.1.2]{Tal}).
\end{proof}

\begin{proof}[Proof of Proposition \ref{p.torus}]
 The upper and lower bounds follow from Propositions \ref{p.r2ub} and \ref{p.r2lb}.
\end{proof}

\section*{Acknowledgment}

We thank Daqian Sun for proofreading part of an early draft, and the referee for suggestions that improved the paper.

\end{document}